\documentclass[a4]{amsart}
\usepackage{datetime,url}
\renewcommand\appendix{\par
  \setcounter{section}{0}
  \setcounter{subsection}{0}
  \setcounter{equation}{0}
  \setcounter{figure}{0}
  \setcounter{table}{0}
  \renewcommand\thesection{Appendix \Alph{section}}
  \renewcommand\theequation{\Alph{section}\arabic{equation}}
  \renewcommand\thefigure{\Alph{section}\arabic{figure}}
  \renewcommand\thetable{\Alph{section}\arabic{table}}
}

\pdfoutput=1
\newtheorem{theorem}{Theorem}[section]
\newtheorem{lemma}{Lemma}[section]

\newcommand{\trev}{}

\newtheorem{proposition}{Proposition}[section]
\newtheorem{corollary}{Corollary}[section]

\theoremstyle{definition}

\newtheorem{example}{Example}[section]

\newtheorem{remark}{Remark}[section]

\usepackage{graphicx}
\usepackage{color}
\usepackage{amsmath}
\usepackage{amsfonts}
\usepackage{amssymb}
\usepackage{mathrsfs}
\newcommand{\be}{\begin{equation}}
\newcommand{\ee}{\end{equation}}

\newcommand{\beqn}{\begin{eqnarray*}}
\newcommand{\eeqn}{\end{eqnarray*}}

\renewcommand{\gg}{g}

\newcommand{\D}{\Delta}

\newcommand{\f}{\varphi}

\newcommand{\R}{{\mathbb R }}

\newcommand{\ba}{\begin{array}}

\newcommand{\ea}{\end{array}}

\newcommand{\beq}{\begin{eqnarray}}

\newcommand{\eeq}{\end{eqnarray}}

\newtheorem{lm}{lemma}
\newcommand{\blem}{\begin{lemma}}
\newcommand{\elem}{\end{lemma}}

\newtheorem{thee}{theorem}

\newtheorem{proo}{proposition}

\newtheorem{co}{corollary}

\newtheorem{rem}{remark}

\newtheorem{deff}{definition}

\newcommand{\bd}{\begin{deff}}

\newcommand{\ed}{\end{deff}}

\newcommand{\bl}{\begin{lm}}

\newcommand{\el}{\end{lm}}

\newcommand{\bp}{\begin{proo}}

\newcommand{\ep}{\end{proo}}

\newcommand{\bt}{\begin{thee}}

\newcommand{\et}{\end{thee}}

\newcommand{\bc}{\begin{co}}

\newcommand{\ec}{\end{co}}

\newcommand{\brm}{\begin{rem}}

\newcommand{\erm}{\end{rem}}

\newcommand{\der}{d}

\usepackage{t1enc}
\def\frak{\mathfrak}

\def\Cal{\mathcal}

\newcommand{\newc}{\newcommand}

\newcommand{\cA}{\mathcal A}

\newcommand{\spA}{\mathrm{span}_{\cA}}

\let\ccdot\cdot
\def\cdot{\hbox to 2.5pt{\hss$\ccdot$\hss}}

\newc{\aR}{\mbox{\boldmath{$ R$}}}
\newc{\aS}{\mbox{\boldmath{$ S$}}}
\newc{\aT}{\mbox{\boldmath{$ T$}}}
\newc{\aW}{\mbox{\boldmath{$ W$}}}

\newcommand{\1}{\mathbb{I}}

\newc{\aK}{\mbox{\boldmath{$ K$}}}
\newc{\aL}{\mbox{\boldmath{$ L$}}}

\newcommand{\CalH}{\mathcal{H}}

\usepackage{amssymb}
\usepackage{amscd}

 \newcommand{\tnab}{\widetilde{\nabla}}

\let\e=\varepsilon
\let\f=\varphi

\newcommand{\hook}{\raisebox{-0.35ex}{\makebox[0.6em][r]
{\scriptsize $-$}}\hspace{-0.15em}\raisebox{0.25ex}{\makebox[0.4em][l]{\tiny
 $|$}}}

\newcommand{\del}{\partial}

\let\s=\sigma

\newc{\obstrn}[2]{B^{#1}_{#2}}

\newcommand{\rpl}                         
{\mbox{$
\begin{picture}(12.7,8)(-.5,-1)
\put(0,0.2){$+$}
\put(4.2,2.8){\oval(8,8)[r]}
\end{picture}$}}

\newcommand{\lpl}                         
{\mbox{$
\begin{picture}(12.7,8)(-.5,-1)
\put(2,0.2){$+$}
\put(6.2,2.8){\oval(8,8)[l]}
\end{picture}$}}

\usepackage{ifthen}
\newcommand{\bbR}{\mathbb{R}}

\newcommand{\sog}{\mathbf{SO}}
\newcommand{\spin}{\frak{spin}}

\newcommand{\soa}{\frak{so}}

\newc{\tensor}[1]{#1}
\newc{\Mvariable}[1]{\mbox{#1}}
\newc{\down}[1]{{}_{#1}}
\newc{\up}[1]{{}^{#1}}

\newc{\JulyStrut}{\rule{0mm}{6mm}}
\newc{\midtenPan}{\mbox{\sf S}}
\newc{\midten}{\mbox{\sf T}}
\newc{\midtenEi}{\mbox{\sf U}}
\newc{\ATen}{\mbox{\sf E}}
\newc{\BTen}{\mbox{\sf F}}
\newc{\CTen}{\mbox{\sf G}}

\def\sideremark#1{\ifvmode\leavevmode\fi\vadjust{\vbox to0pt{\vss
 \hbox to 0pt{\hskip\hsize\hskip1em
 \vbox{\hsize3cm\tiny\raggedright\pretolerance10000
 \noindent #1\hfill}\hss}\vbox to8pt{\vfil}\vss}}}%
                        
\newcommand{\Span}{\mathrm{Span}}

\numberwithin{equation}{section}

\newcounter{romenumi}
\newcommand{\labelromenumi}{(\roman{romenumi})}

\newcommand{\bma}{\begin{pmatrix}}
\newcommand{\ema}{\end{pmatrix}}

\begin{document}
	\setlength{\leftmargini}{7mm}
 \newcommand{\marg}[1]{\marginpar{\scriptsize  #1}}
 \newcommand{\thomas}[1]{\marg{\color{blue} T. #1}}
\bibliographystyle{abbrv}
\title{Explicit  ambient metrics and  holonomy} 
\vskip 1.truecm 
\author{Ian M.~Anderson} \address{Department of Mathematics and Statistics, Utah State University, Logan Utah, 84322, USA}
\email{\tt Ian.Anderson@usu.edu}
\author{Thomas Leistner}\address{School of Mathematical Sciences, University of Adelaide, SA 5005, Australia} \email{\tt thomas.leistner@adelaide.edu.au}
\author{Pawe\l~ Nurowski} \address{Centrum Fizyki Teoretycznej PAN 
Al. Lotnik\'{o}w 32/46 
02-668 Warszawa, Poland}
\email{nurowski@cft.edu.pl} \thanks{
Support for this research was provided by National Science Grant ACI 1148331SI2-SSE (IA), by
the
 Australian Research
Council via the grants FT110100429 and DP120104582,
and by the Polish National Science Center (NCN)  via  DEC-2013/09/B/ST1/01799.
%
}

\begin{abstract}We present three large classes of examples of conformal structures for which the equations for the Fefferman-Graham ambient metric to be Ricci-flat are linear PDEs, which we solve explicitly. These explicit solutions enable us to discuss the holonomy of the corresponding ambient metrics. Our examples include conformal pp-waves and, more importantly,  conformal structures that are defined by generic rank $2$ and $3$ distributions in respective dimensions $5$ and $6$. The corresponding explicit Fefferman-Graham ambient metrics provide a 
large class of metrics with holonomy equal to the exceptional non-compact Lie group $\mathbf{G}_2$ as well as ambient metrics with holonomy contained in $\mathbf{Spin}(4,3)$.
\end{abstract}
\maketitle 
\setcounter{tocdepth}{1}
\newcommand{\gat}{\tilde{\gamma}}
\newcommand{\Gat}{\tilde{\Gamma}}
\newcommand{\thet}{\tilde{\theta}}
\newcommand{\Thet}{\tilde{T}}
\newcommand{\rt}{\tilde{r}}
\newcommand{\st}{\sqrt{3}}
\newcommand{\kat}{\tilde{\kappa}}
\newcommand{\kz}{{K^{{~}^{\hskip-3.1mm\circ}}}}
\newcommand{\gpe}{{g_{_{PE}}}}
\newcommand{\upe}{{\mathcal U}_{_{PE}}}
 

\newcommand{\vect}[1]{\mathbf{#1}}
\renewcommand{\k}{\mathbf{k}}
\renewcommand{\l}{\mathbf{l}}
\renewcommand{\e}{\mathrm{e}}
\renewcommand{\f}{\mathbf{f}}
\newcommand{\bprf}{\begin{proof}}
\newcommand{\eprf}{\end{proof}}


\section{Introduction}
\trev{
Let $g_0$ be a smooth semi-Riemannian metric  on an $n$-dimensional 
manifold $M$, with $n\ge 2$ and  with local coordinates $(x^i)$.
A {\em pre-ambient metric for $g_0$} is, by definition, a smooth semi-Riemannian metric $\tilde g$ 
of the form
\begin{equation}\label{ambient-intro}
	\tilde{g}=2\, dt  d(\varrho t)+t^2g(x^i,\varrho) 
\end{equation}
defined on the 
{\em  ambient space}
$$\tilde{M}=(0,+\infty)\times M\times(-\epsilon,\epsilon),\qquad \epsilon>0,$$
with coordinates $(t,x^i,\varrho)$
and
such that 
$g(x^i,\varrho)_{|\varrho=0}=g_0(x^i)
$.
We call a pre-ambient metric an {\em ambient metric} if 
 \begin{equation}\label{fg-eqs}
	Ric(\tilde g) = 0.
\end{equation}
It is a fundamental observation of  Fefferman and Graham \cite{fefferman/graham85,fefferman-graham07}
that conformally equivalent metrics have the same ambient metrics. More precisely, given a metric $g_0$ with ambient metric $\widetilde{g}$  and a conformally equivalent metric $\hat g_0$
 in the  conformal class  $[g_0]=\{\mathrm{e}^{2\phi}g_0\mid \phi \in C^\infty(M)\}$ of $g_0$, an ambient metric for $\hat g_0$ can be obtained by pulling back $\widetilde g$ by a diffeomorphism of $\widetilde M$. The importance of the ambient metric is that the semi-Riemanian invariants  of $\tilde g$ 
        give rise to conformal invariants of $g_0$. 
        
Given a pre-ambient metric as in \eqref{ambient-intro} we call the system of PDEs for an unknown $\varrho$-dependent family of metrics $g(x^i,\varrho)$ the {\em Fefferman-Graham equations}. Their explicit form is given in \cite[eq. 3.17]{fefferman-graham07}. 
 In general we assume  ambient metrics to be smooth, however, as we attempt in our examples to find the most general solutions to the Fefferman-Graham equations,  we will also present Ricci-flat metrics of the form \eqref{ambient-intro}  that are defined only for  $\varrho \ge 0$ and  differentiable only to finite order  at $\varrho=0$. By a slight abuse of terminology, we will  call them {\em non-smooth ambient metrics}.

Regarding the existence and uniqueness of ambient metrics, Fefferman and Graham \cite{fefferman/graham85,fefferman-graham07}
 have shown  the following:
If  $g_0$ is a smooth metric in odd dimension $n$, then there is always a smooth pre-ambient metric with 
 \begin{equation}
 \label{nabkric}
 \widetilde{\nabla}^kRic(\widetilde g)|_{\varrho=0}=0,
 \end{equation}
 for  all $k=0,1,\ldots\ $, and in which $\widetilde{\nabla}$ is the Levi-Civita connection of $\widetilde{g}$. The proof of this result proceeds by constructing  $g(x^i,\varrho)$ from the initial conditions as a power series in $\varrho$. The  coefficients of this power series are determined by  equation \eqref{nabkric} and thus yield a formal solution to the Fefferman-Graham equations \eqref{fg-eqs}. One can show that the obtained formal power series is convergent whenever $g_0$ is analytic, and hence, in this  case one gets an ambient metric as defined above. For analytic $g_0$ this is the {\em unique analytic} ambient metric. 
 
 When  $n>2$ is even, the situation is more subtle, and in general one can only find pre-ambient metrics satisfying the condition  \eqref{nabkric} for $k=0, \ldots, \frac{n}{2}-1$. 
For \eqref{nabkric} to be satisfied beyond the order $k=\tfrac{n}{2}-1$ one checks if a certain symmetric $(2,0)$-tensor, the so-called {\em Fefferman-Graham obstruction tensor} $\mathcal O_{ij}(g_0)$, vanishes. Fefferman and Graham showed that the vanishing of $\mathcal O_{ij}(g_0)$ is equivalent to the existence of a pre-ambient metric satisfying condition \eqref{nabkric} up to all orders $k$, and hence for analytic $g_0$'s, equivalent to the existence of an ambient metric in  our sense. However, when $n$ is even and  $\mathcal O_{ij}(g_0)$ vanishes,  analytic ambient metrics are not unique.
 }
        
        Unfortunately, for general metrics $g_0$ 
	the  Fefferman-Graham equations \eqref{fg-eqs}
	are a complicated non-linear system of partial differential equations for $g(x^i,\varrho)$ 
	which in general cannot be solved explicitly. This makes it very hard to find explicit ambient metrics. 
	Indeed, prior to the results of  \cite{leistner-nurowski08,leistner-nurowski09}, the only known explicit examples
	of ambient metrics were related to
	 Einstein metrics $g_0$ or certain products of Einstein metrics
	\cite{leistner05a,Leitner05,GoverLeitner09}. 
	In fact, if $g_0$ is an Einstein metric, i.e., $Ric(g_0)=\Lambda g_0$, then an ambient metric is given by
\begin{equation}
\label{einstein-ambient}
\tilde{g}=2\der t\der(\varrho t)+t^2 \Big(1+\frac{\Lambda\varrho}{2(n-1)}\Big)^2g_0.\end{equation}
In other words, for an Einstein metric $g_0$, a solution to the Fefferman-Graham equations is given by $g(x^i,\varrho)=(1+\frac{\Lambda}{2(n-1)}\varrho)^2g_0$.
	
	The goal of this paper 
	is to present 
	 some interesting  classes of examples of  metrics $g_0$
	for which, quite remarkably, 
	{\it the Fefferman-Graham equations reduce to linear PDEs}. 
The 	more general class of metrics with 
	linear Fefferman-Graham equations will be described in our subsequent article~\cite{ipt2}. 	
	Our approach for solving the Fefferman-Graham equations in  our examples, either in closed form or
	by standard power series methods,  is based on a particular feature of the equations. The equations are of second order, but, since they have a singular point at $\varrho=0$,    the first $\varrho$-derivative of $g(x,\varrho)$ at $\varrho=0$ cannot be  chosen freely. Instead, 
	it is determined as $g^\prime|_{\varrho=0}=2P$, where $P$ is the \emph{Schouten tensor}
$$P~=~\frac{1}{n-2}\big(\, Ric(g_0)~-~\frac{Scal(g_0)}{2(n-1)}~ g_0\, \big)$$
of $g_0$. The common feature of the  considered examples is that   the Schouten tensor of $g_0$ is nilpotent (when considered as $(1,1)$-tensor), and we postulate the following ansatz for $g(\varrho, x^i)$
\[g(x^i,\varrho)=g_0+h(x^i,\varrho),\]
where  $h(x^i,\varrho)$ is required to have the {\em same image and kernel as $P$} for all $\varrho$ (again considered as $(1,1)$-tensors). Then the Fefferman-Graham equations turn out to be  {\em linear} second order PDEs for the unknown symmetric tensor $h=h(x^i,\varrho)$.

	The conformal structures presented here are of three different types. 
The first class of examples, 
	which are really a  \emph{prototype}  for all the conformal structures 
	for which the Fefferman-Graham equations become linear, are the so-called \emph{conformal pp-waves}.
	These are presented  in Section~\ref{secpp}.  We extend our results from \cite{leistner-nurowski08} to other signatures \trev{and by determining all solutions to \eqref{fg-eqs}. Hence, we find ambient metrics which, when $n$ is odd, are  defined for $\varrho\ge 0$ and only $\frac{n}{2}-1$-times differentiable  in $\varrho$ at $\varrho=0$. When $n$ is even, the obtained ambient metrics are only defined on the domain $\varrho >0$ as they contain logarithmic terms.} An upper bound for the holonomy algebras of the ambient metrics is determined.   
	
The second class of examples, discussed in Section \ref{secg2},   arises from   conformal 
	structures in signature $(3,2)$ defined by generic rank 2 distributions $\mathcal D$ in dimension 5 \cite{nurowski04}. These distributions can be brought into normal form
	\begin{equation}
	\label{distribution} 
	{\mathcal D}={\rm Span}\Big(\partial_q,\partial_x+p\,\partial_y+q\,\partial_p+F\,\partial_z\Big),\end{equation}
	where  our coordinates on $\R^5$ are $(x,y,z,p,q) $, we write $\del_q=\frac{\partial}{\partial q}$ etc.,  and $F=F(x,y,z,p,q) $ is a smooth function with $\partial_q^2(F)\not=0$. 
The most remarkable aspect of the  associated conformal structures is that their normal conformal Cartan connection reduces to the 
exceptional $14$-dimensional, simple,  non-compact Lie group $\mathbf{G}_2$ (as shown in \cite{nurowski04}) and hence have their  
 conformal holonomy\footnote{We 
 point out that all of our considerations are {\em local} in the sense that we work on simply connected manifolds. This implies that  the holonomy groups we encounter are connected and hence can be identified with their Lie algebras. Moreover, throughout the paper  $\mathfrak{g}_2$ and  $\mathbf{G}_2$ refer   to the {\em split real form} of the exceptional simple Lie algebra of type $G_2$ and to the corresponding connected Lie group in $\mathbf{SO}(4,3)$.}  
contained in this group. In \cite{leistner-nurowski09} we showed that for  distributions defined by
$F=q^2+\sum_{i=0}^6 a_ip^i+b z$,
with constants $a_i$ and $b$, 
  also the holonomy of the associated analytic ambient metric is contained in, and generically equal to, $\mathbf{G}_2$. The results in \cite{leistner-nurowski09} were subsequently generalized  in 
 \cite{graham-willse11} to  {\em all}   conformal structures defined by  analytic distributions $\mathcal D$ as in \eqref{distribution}, by showing that their unique analytic
  ambient metric admits a certain parallel $3$-form (or equivalently, a parallel non-null spinor field) whose stabilizer is   $\mathbf{G}_2$. This result shows that not only the conformal holonomy but also   the holonomy of the ambient metric is  contained in $\mathbf{G}_2$. Also in  \cite{graham-willse11}  a sufficient criterion was given for the holonomy of the ambient metric being equal to $\mathbf{G}_2$, however, the examples in
  \cite{leistner-nurowski09} with holonomy equal to $\mathbf{G}_2$ do {\em not} satisfy this criterion, and in general the criterion is very difficult to check.
   	Here we generalize the results of \cite{leistner-nurowski09}  to a much larger class. In fact, for all conformal classes associated to distributions $\mathcal D$ given by a function $F$ of the following form, 
	\begin{equation}
	\label{fh} 
F=q^2+f(x,y,p)+h(x,y)z,\end{equation}
	the Fefferman-Graham equations turn out to be linear second order PDEs. For such conformal classes, in Section \ref{g2sec} we will  show the following: 
\begin{enumerate}\setlength{\labelsep=2pt}
	\item \trev{We provide the most general formal   solutions (in terms of power series) to the Fefferman-Graham equations, including those that are only twice differentiable in $\varrho$ at $\varrho=0$.}
	\item We derive a linear {\em first order} system that has the same (analytic) solutions as the  Fefferman-Graham equations.
	\item Using this first order system, we derive explicit formulae for the $3$-form and the parallel spinor field for the analytic ambient metric.
	\item For particular choices of $f$ and $h$ we give  explicit (closed form) ambient metrics.
	\item We give an explicit, easy to check criterion  that ensures that the ambient metric has holonomy {\em equal} to $\mathbf{G}_2$.
	\end{enumerate}
	\trev{The first order system in (ii) is crucial as it gives us a tool, leaving the conformal context aside, to construct
	metrics in dimension $7$ with holonomy contained in $\mathbf{G}_2$. Together with the } 
 criterion in (v) it enables us to construct  a large class of metrics  with holonomy {\em equal to} $\mathbf{G}_2$ and   depending on the functions $f=f(x,y,p)$ and $h=h(x,y)$. 
	The examples we obtain include, but are not restricted to, ambient metrics 
 which  are  polynomial functions 
	in the ambient space variable $\varrho$ of {\it any} prescribed order\footnote{We should mention that during the preparation of this paper, the article \cite{Willse14} appeared in which it was shown that the conformal class arising from the distribution found by Doubrov and Govorov \cite{DoubrovGovorov13} with $F=q^{1/3}+y$, defined on a domain with $q>0$, has a (non-polynomial) ambient metric with holonomy equal to $\mathbf{G}_2$.}. 
	\trev{Furthermore, we discuss the holonomy of the non-smooth solutions and surprisingly find metrics with holonomy equal to $\mathbf{G}_2$ amongst them.}

	The third class of examples, which are altogether new and presented in Section \ref{secspin43}, are derived from Bryant's 
	 conformal structures   \cite{bryant06}. They  are naturally associated with generic rank 3-distributions in 
	dimension 6, 
	i.e., rank $3$ distributions $\mathcal D$ on a $6$-manifold $M$ with
	$[{\mathcal D},{\mathcal D}]+{\mathcal D}= T M$.
	According to  \cite{bryant06}, to such distributions one can associate a conformal structure in dimension $6$, whose {\em conformal holonomy} is now contained in $\mathbf{Spin}(4,3)$.
	In this paper we consider a certain family of such distributions locally  on $\R^6$, which is given in terms of a differentiable function  $f=f(x^1,x^2,x^3)$ of {\em only  three variables $(x^1,x^2,x^3)$} satisfying generiticity condition $\frac{\partial f}{\partial x^3}\neq 0$. Explicitly, the distributions from the family are given by	\begin{equation}\label{33dis}
	{\mathcal D}=\Span\left(\,\frac{\partial}{\partial x^3}-x^2\frac{\partial}{\partial y^1},\,\,\,\frac{\partial}{\partial x^1}-f(x^1,x^2,x^3)\,\frac{\partial}{\partial y^2},\,\,\,\frac{\partial}{\partial x^2}-x^1\frac{\partial}{\partial y^3}\,\right), \end{equation}
where $(x^1,x^2,x^3, y^1,y^2,y^3)$ are coordinates on $\R^6$.
We show that for conformal structures defined by $\mathcal D$ with $f=f(x^1,x^3)$ or $f=f(x^2,x^3)$ the 
Fefferman-Graham obstruction tensor $\mathcal O_{ij}(g_0)$
	vanishes. However,   generically they are not conformally Einstein. 
With the vanishing of the obstruction tensor, we are assured that for the conformal classes associated with such $\mathcal D$'s, \trev{the Fefferman-Graham equations have an analytic solution, whenever $f$ is analytic.} Again the Fefferman-Graham equations turn out to be  {\em linear}, which enables us to find the ambient metrics explicitly. To our knowledge, these are the first examples of ambient metrics for Bryant's conformal classes. However, in this case, as the general theory for even dimensions predicts, they  are not unique anymore. Now, new Ricci flat ambient metrics are given and parametrized by a symmetric, trace-free  $(2,0)$-tensor~$\varrho^3Q$.

For the ambient metrics of Bryant's conformal classes defined by 
$f=f(x^1,x^3)$ or $f=f(x^2,x^3)$ we give an upper bound for the  holonomy of the ambient metric by showing that it has to be contained in Poincar\'{e} group in signature $(3,3)$,
\[ \mathbf{PO}(3,3)=\mathbf{SO}(3,3)\ltimes \mathbb{R}^{3,3},\]
which is the stabilizer of a  decomposable $4$-form defining an invariant totally null $4$-plane.
This implies that these ambient metrics cannot have holonomy {\em equal to} $\mathbf{Spin}(4,3)$.
Furthermore,  we show that for ambient metrics with $Q=0$ the holonomy reduces further to
	 \[ \mathbf{PO}(3,2)=\mathbf{SO}(3,2)\ltimes \mathbb{R}^{3,2},\]
	 which in turn is {\em contained} in $\mathbf{Spin}(4,3)$. Note that the representations of the Poincar\'{e} groups in $\mathbf{SO}(4,4)$ are {\em not} the usual  representations as stabilizers of null vectors.
	 	 Moreover we compute 
	  three examples of  ambient metrics, including one with $Q\not=0$. In the example with $Q\neq 0$,  the  ambient holonomy is equal to  $\mathbf{PO}(3,3)$, and hence, is {\em not contained} in $\mathbf{Spin}(4,3)$. The other two examples of ambient metrics  
have $Q=0$, one with  holonomy equal to $\mathbf{PO}(3,2)$, and the other with the $7$-dimensional Heisenberg group as  holonomy.
We believe that these ambient metrics reveal some interesting $4$-form geometry in signature $(4,4)$, which suggests further investigation.

Finally, we  give an explicit example of a conformal structure, 
	associated with a rank 3-distribution as in \eqref{33dis} having $f=x^3+x^1 x^2+(x^2)^2+(x^3)^2$,
	for which the obstruction tensor does {\it not} vanish. This example shows that the vanishing of the obstruction tensor $\mathcal O$ for the conformal classes associated with $f=f(x^1,x^3)$ or $f=f(x^2,x^3)$ is exceptional and is associated with the chosen symmetry, rather than the general feature of  distributions with $f$ depending on  the three variables $(x^1,x^2,x^3)$. Again,  many open questions remain, for example:  What are the conditions on $f$ that are equivalent to the vanishing of the obstructions tensor, and in case of vanishing obstruction, how many ambient metrics are there with  holonomy (contained in) $\mathbf{Spin}(4,3)$?

	Many of the calculations for this paper were performed using the
	Maple {\it DifferentialGeometry} package. 
	Sample worksheets, illustrating the results of this paper,  can 
	be found at \url{http://digitalcommons.usu.edu/dg_applications/}. 
\\[2mm]
{\bf Acknowledgements.} We thank Travis Willse for comments on the first version of this article.

\section{Generalized conformal pp-waves and their ambient metrics}\label{secpp}
\subsection{Generalized pp-waves}
In this section we provide explicit formulae for ambient metrics for  conformal structures on pseudo-Riemannian manifolds which, in Lorentzian signature, are known as  \emph{pp-waves}. The formulae presented here generalize our  results on conformal pp-waves  in \cite{leistner-nurowski08}.
Lorentzian  pp-wave metrics are  defined on an open set $\mathcal{U}\subset \bbR^n$, with coordinates $(x^i, u,v)$, $i=1,2,\dots,n-2$, where they assume the form
$$
g_H=2\der u(\der v+H(x^i,u)\,\der u)+\sum_{i=1}^{n-2}(\der x^i)^2.$$
Here $H=H(x^i,u)$ is a real function on $\mathcal U$ which specifies a pp-wave.
For Lorentzian pp-waves of odd dimension  we have determined the unique  ambient metric that is 
analytic in $\varrho$
 in \cite{leistner-nurowski08} as
\[
\widetilde{g}_H
= 2 \der \left(t \varrho  \right) \der t
 + t^2 \Big(g
+ \big( \sum_{k=1}^{\infty} \frac{\D^kH}{k ! p_k }\varrho^k \big)\ 
\der u^2\Big),\]
where $p_k=\prod_{j=1}^{k}(2j-n)$ and  
 $\Delta=\sum_{i=1}^{n-2}\del_i^2$ is  the flat Laplacian.
When  
 $n$ is even, this formula also gives analytic ambient metrics if one assumes that  $\Delta^\frac{n}{2} H=0$.
 Here we  generalise this result to higher signature and, more importantly, determine all solutions to the Fefferman-Graham equations.

Let $1\le p\le n/2$ be a natural number
and let the indices  run as follows: $A,B,C \ldots \in\{1, \ldots n-2p\}$
$a,b,c \ldots \in\{1, \ldots p\} $. Let 
\[\Cal U\subset \R^n\ni(x^1, \ldots , x^{n-2p}, u^1, \ldots , u^p, v^1, \ldots, v^p)\] be an open set, and $H_{ac}=H_{ca}$ smooth functions on $\Cal U$ satisfying  and 
$\frac{\del}{\del v^b}(H_{ac}) =0$.  
Then we  define a {\em generalized pp-wave}  in signature $(p,n-p)$ (for short {\em pp-wave}),  as follows:
\be
\label{pp}
\gg_H
\
=
\
2 du^a( \delta_{a b} dv^b +  H_{a b}du^b)+ \delta_{AB}dx^Adx^B 
,
\ee
where we denote  $H=(H_{ac})_{a,c=1}^p$.
Such pp-waves admit $p$ parallel null vector fields $\frac{\del}{\del v^a}$. 
The scalar curvature of pp-waves vanishes and their Ricci tensor is given as
\[Ric(\gg_H)=-\Delta(H_{ac})du^a du^c, \]
 with $\Delta=\delta^{AB}\del_A\del_B$  the flat Laplacian with respect to the $x^A$-coordinates.
 Hence, the Schouten tensor of a pp-wave  is given by
 \begin{equation}
 \label{ppschouten}
 P=-\tfrac{1}{n-2}\,\Delta(H_{ac})  \, du^a du^c,\end{equation}
 and therefore is $2$-step nilpotent (i.e. $P^2=0$),  when considered as a $(1,1)$-tensor.
 
 \subsection{Analytic and non-analytic ambient metrics for pp-waves}
The special form of the Schouten tensor in \eqref{ppschouten} of a  metric $g_H$ in \eqref{pp} suggests the following ansatz  for $g(x^A,u^b,v^c,\varrho)$ 
(which we will motivate further in \cite{ipt2})
for the ambient metric  of the conformal class $[g_H]$:
\begin{equation}\label{ambientpp}\widetilde{\gg}_{H,h}=2\der(\varrho t)\der t+t^2(\gg_H+2h_{ac}du^a d u^c)\end{equation} 
with unknown   functions  $h_{ac}=h_{ac}(\varrho, x^A,u^b))$ not depending on the coordinates $v^1, \ldots, v^p$.
Then, computations  completely analogous to \cite{leistner-nurowski08}  prove:
\begin{theorem}\label{th1}
Let  $\gg_H$ be a  pp-wave metric as in \eqref{pp} defined by functions $H_{ac}=H_{ac}(x^A,u^b)$. Then the 
metric $\widetilde{
\gg}_{H,h}$ in 
\eqref{ambientpp} is Ricci-flat if and only if each of the
 unknown   functions  $h_{bd}=h_{bd}(\varrho, x^A,u^b)$ 
satisfies  the linear second order linear PDE 
\begin{equation}\label{hsys}
2\varrho  h^{\prime\prime}_{bd}+(2-n) h^\prime_{bd}
-\Delta(H_{bd}+h_{bd})
=0,
\end{equation}
where the prime denotes the $\varrho$-differentiation. In particular, the metric $\widetilde{
\gg}_{H,h}$ is an ambient metric for the conformal class $[\gg_H]$ given by the pp-wave $\gg_H$ if, in addition, the initial conditions 
\[h_{bd}(\varrho,x^A,u^b)|_{\varrho=0}=0\]
are satisfied.
\end{theorem}

In \cite{leistner-nurowski08} we have seen how
equation \eqref{hsys} can be solved by standard power series expansion, noticing that its indicial exponents are $s=0$ and $s=n/2$. Here we give a more general existence  result which includes non-analytic solutions and reflects the non-uniqueness of the ambient metric in even dimensions by explicitly showing the ambiguity at order $n/2$.

\begin{theorem}\label{pptheo}
Let $
\Delta=\delta^{AB}\del_A\del_B$ be  the flat Laplacian in $(n-2p)$ dimensions and $H=H(x^A,u^b)$.
\begin{enumerate}\setlength{\labelsep=2pt}
\item When $n$ is odd, the most general solutions $h=h(\varrho,x^A,u^b)$ to equation
 \begin{equation}
\label{fg1neutral4}
2\varrho  h^{\prime\prime}+(2-n) h^\prime
-\Delta(H+h)
=0,
\end{equation}
with $h(\varrho)\to 0$  when $\varrho\downarrow 0$ are given as
\be
h=\sum_{k=1}^{\infty} \frac{\Delta^kH}{k!\prod_{i=1}^k(2i-n)}\varrho^k
+
\varrho^{n/2}\Big(\alpha+\sum_{k=1}^\infty \frac{\Delta^k\alpha}{k!\prod_{i=1}^k(2i+n)}\varrho^k\Big),\label{pro}\ee
where $\alpha=\alpha(x^A,u^b)$ is an arbitrary function of its variables.
In particular, when $H$ is analytic, a unique solution that is analytic in $\varrho$ in a neighbourhood of $\varrho=0$ with $h(0)=0$ is given  by 
 $\alpha\equiv 0$.
\item
When $n=2s$ is even, the most general solutions $h$ to equation (\ref{fg1neutral4}) with $h(\varrho)\to 0$  when $\varrho\downarrow 0$ are given by
\begin{eqnarray}
\label{pre}
h&=&
\sum_{k=1}^{s-1}
 \frac{\Delta^kH}{k!\prod_{i=1}^k(2i-n)}\varrho^k
 + 
 \varrho^{s}\Big(\alpha+\sum_{k=1}^\infty \frac{\Delta^k\alpha}{k!\prod_{i=1}^k(2i+n)}\varrho^k\Big)
  \\
\nonumber &&+\
c_n\varrho^{s}\left(
 \sum_{k=0}^\infty
\left(\log(\varrho)
- q_k\right)\frac{\Delta^{s+k}H}{k!\prod_{i=1}^k(2i+n)}\varrho^k\right)
\\&&{ }
+ c_n\varrho^{s} \ Q\ast 
\sum_{k=0}^\infty
\frac{\Delta^{s+k}H}{k!\prod_{i=1}^k(2i+n)}\varrho^k,
\nonumber
\end{eqnarray}
where $\alpha=\alpha(x^A,u^b)$ and $Q=Q(x^A,u^b)$ are arbitrary functions of their variables, $\ast$ denotes the convolution of two functions with respect to the $x^A$-variables, and the constants are given as
\[ c_n=-
\frac{1}{
(s-1)! \prod_{i=0}^{s-1}(2i-n)},\ \  \
 q_0 = 0,\ \ \ q_k = \sum_{i=1}^k\frac{n+4i}{i(n+2i)},\ \text{  for  $k=1, 2, \ldots $.}\]
In particular, only when $\Delta^{s}H\equiv 0$  are there solutions  that are analytic in $\varrho$ in a neighbourhood of $\varrho=0$ and with $h(0)=0$. These solutions are not unique but are parametrized by analytic functions $\alpha$.
\end{enumerate}
\end{theorem}
\begin{proof}
Let $h$ be a solution to \eqref{fg1neutral4} with initial conditions 
 $h(\varrho,x^A,u^b)_{|\varrho=0}\equiv 0$. Then 
 $F= H+h$
  is a solution to the homogeneous equation 
 \be
\varrho F''+(1-\tfrac{n}{2})F'-\tfrac12 \Delta F=0,\label{feqhom}
\ee
with the initial condition $ F(\varrho,x^A,u^b)_{|\varrho=0}= H(x^A,u^b)$. 
Let $\hat F$ and $\hat H$ denote the Fourier transforms with respect to the $x^A$-variables of $F$  and $H$. Recall  that $\widehat{\Delta F}=-\|x\|^2\hat F$, we easily find
the Fourier transform of  equation \eqref{feqhom} to be
 \be
\varrho \hat{F}''+(1-\tfrac{n}{2})\hat{F}'+\frac{\|x\|^2}{2} \hat{ F}=0,\label{feqhomode}
\ee
with initial condition  $ \hat F(\varrho,x^A,u^b)_{|\varrho=0}= \hat H$.
For each point $(x^A,u^b)\in \mathbb{R}^{n-p}$, which we fix for the moment, 
 this is an ODE in $\varrho$ whose solutions can be determined by the Frobenius method (see for example \cite[Section 6.3.3]{coddington-carlson97}). The indicial polynomial of  \eqref{feqhomode} is  $\lambda(\lambda-1)-(\frac{n}{2}-1)\lambda$ and hence has zeros $0$ and $\frac{n}{2}$, which forces us to distinguish the cases when $n$ is odd and $n$ is even.
 
For $n$ odd the general solution to \eqref{feqhomode} is given by
\be \label{proode1}
\hat F = a \varrho^{n/2} F_1 + b F_2,
\ee
for arbitrary constants $a$ and $b$, and with $F_1$ and $F_2$  given by
\begin{eqnarray}
\label{f1}F_1(x^A,\varrho)&=&1+\sum_{k=1}^\infty \frac{(-1)^k\|x\|^{2k}}{k!\prod_{i=1}^k(2i+n)}\varrho^k\\
\label{f2}F_2(x^A,\varrho)&=& 1+\sum_{k=1}^{\infty} \frac{(-1)^k\|x\|^{2k}}{k!\prod_{i=1}^k(2i-n)}\varrho^k.
\end{eqnarray}
Hence, when we vary $(x^A,u^b)$ we get that the general solution to \eqref{feqhomode} is given by
\be \label{proode}
\hat F = \alpha \varrho^{n/2} F_1 + \beta F_2,
\ee
where $\alpha=\alpha(x^A,u^b)$ and $\beta=\beta(x^A,u^b)$ are arbitrary functions of their variables.
The initial conditions $\hat F|_{\varrho=0}=\hat H$, or more precisely $\hat F\to \hat H$ if $\varrho\downarrow 0$, imposes $\beta =\hat{ H}$. Hence, we obtain
\begin{eqnarray*}
 \hat F & = & \sum_{k=0}^{\infty} \frac{(-1)^k\|x\|^{2k}\hat H}{k!\prod_{i=1}^k(2i-n)}\varrho^k
 + \sum_{k=0}^\infty \frac{(-1)^k\|x\|^{2k}\alpha}{k!\prod_{i=1}^k(2i+n)}\varrho^k
 \\
 &=&
  \sum_{k=0}^{\infty} \frac{ \widehat{\Delta^k H}}{k!\prod_{i=1}^k(2i-n)}\varrho^k
 + \sum_{k=0}^\infty \frac{\widehat{\Delta^k \check{\alpha}}}{k!\prod_{i=1}^k(2i+n)}\varrho^k,
 \end{eqnarray*}
taking into account that 
$ (-1)^k \|x\|^{2k}\hat{ H}  =\widehat{\Delta^k H}$ and denoting by $\check{\alpha}$ the inverse Fourier transform  of $\alpha$. Applying the inverse Fourier transform to this expression gives us
\[F=  \sum_{k=0}^{\infty} \frac{ \Delta^k H}{k!\prod_{i=1}^k(2i-n)}\varrho^k
 + \sum_{k=0}^\infty \frac{\Delta^k \check{\alpha}}{k!\prod_{i=1}^k(2i+n)}\varrho^k,
 \]
 and thus the formula in the case $n$ odd.

If $n=2s$ is even,  the general solution to \eqref{feqhomode} is given by
\be
\hat F = \alpha \varrho^{s} F_1 + \beta \Big( 
 c_n(-1)^{s}\|x\|^n\varrho^{s}
\log(\varrho)F_1 + F_3 \Big)
\ee
where $\alpha=\alpha(x^A,u^b)$ and $\beta=\beta(x^A,u^b)$ are arbitrary functions of their variables, $c_n$ is the constant defined  in the statement of the  theorem,  $F_1$ is defined in \eqref{f1} and $F_3$ can be computed as
\[
F_3= 1+ \sum_{k=1}^{s-1}
 \frac{(-1)^{k}\|x\|^{2k}}{k!\prod_{i=1}^k(2i-n)}\varrho^k
 - c_n(-1)^{s}\|x\|^n\varrho^{s}
\left(  \sum_{k=1}^\infty\frac{(-1)^{k}\|x\|^{2k}  }{k!\prod_{i=1}^k(2i+n)}\hat{q}_k \varrho^k\right).
\]
Here $\hat{q}_0=\hat{q}_0(x^A,u^b)$ is an arbitrary function of the  $x^A$'s and the $u^c$'s and 
 the remaining $\hat q_k$ are defined as
\[ \hat{q}_k =\hat{q}_0+ \sum_{i=1}^k\frac{n+4i}{i(n+2i)},\ \ \text{ for }k=1, 2, \ldots.\]
Again, the initial conditions $\hat F|_{\varrho=0}=\hat H$, or more precisely $\hat F\to \hat H$ when $\varrho\downarrow 0$, imposes $\beta =\hat{ H}$. Applying the inverse Fourier transform to our solutions $\hat F$, and taking into account that 
$ (-1)^k \|x\|^{2k}\hat{ H}  =\widehat{\Delta^k H}$ and that 
the convolution  satisfies $\widehat{Q\ast \Delta H}=2\pi^\frac{n-2}{2}\widehat {Q}\;\widehat{\Delta H}$,
 yields the formula in the theorem with the function  $Q$ in the theorem given by the inverse  Fourier transform of $2\pi^\frac{n-2}{2}\hat{q}_0$.
\end{proof}

\begin{remark} Theorems \ref{th1} and \ref{pptheo}
have the following  consequences:
\begin{enumerate}\setlength{\labelsep=2pt}
\item For $n$ odd, the solutions $h_{ac}$ given  in  Theorem \ref{pptheo} by analytic functions $H_{ac}$  defining the pp-wave metric $g_H$  in \eqref{pp} provide us with the  unique analytic  ambient metric \begin{equation}\label{anaambientpp}
\widetilde{\gg}_H=2\der(\varrho t)\der t+t^2\Big(\gg+2\Big(
\sum_{k=1}^{\infty} \frac{\Delta^kH_{ac}}{k!\prod_{i=1}^k(2i-n)}\varrho^k\Big)
du^a d u^c\Big)
\end{equation}
for the conformal class of $g_H$.
\item Theorem \ref{pptheo}  also shows  that, when $n$ is odd, the ambiguity introduced by the arbitrary function $\alpha$ gives only {\em non-analytic} solutions, as guaranteed by the uniqueness statement in the Fefferman-Graham result. In contrast, when $n$ is even, the ambiguity coming from the function $\alpha$ adds an analytic part to a solution and, in case of $\Delta^{n/2}H_{ac}=0$, gives new {\em analytic}  ambient metrics.\end{enumerate}
\end{remark}
We conclude this section by giving an example of a Lorentzian pp-wave (i.e., with $p=1$), which, in even dimensions, does not satisfy the sufficient condition $\Delta^{n/2}H=0$ for our ansatz to give an analytic ambient metric, but which however admits non-analytic ambient metrics.
\begin{example}\label{BesselExpl}
For $k\in \mathbb{R}$ we consider the $n$-dimensional Lorentzian pp-wave 
\[
g=2\der u\der v+\e^{2k x^1}\der u^2+\sum_{A=1}^{n-2}(\der x^A)^2,\]
i.e., with $H=\frac{1}{2}\e^{2k x^1}$, 
which is a product of a $3$-dimensional pp-wave and the flat Euclidean space of dimension $n-3$. When $k=0$ this is the flat metric which has a flat ambient metric. Hence we assume $k\not=0$ from now on. Note that, when $n$ even,  this metric does not satisfy the condition $\Delta^{n/2}H=0$,  which ensures the existence of an analytic ambient metric.
However, for  finding an ambient metric (non-analytic when $n$ even) we make the  ansatz with $h=\frac{1}{2}\left(\phi(\varrho)-1\right)\e^{2k x^1}$ for a function $\phi$  of $\varrho$ with $\phi(0)=1$ and set
\[
\widetilde{\gg}
\ = \
 2\der(\varrho t)\der t+t^2(\gg+2hdu^2)
 \ =\ 
 2\der(\varrho t)\der t+t^2\left(
 2\der u\der v+\phi(\varrho)\e^{2k x^1}\der u^2+\sum_{i=1}^{n-2}(\der x^i)^2\right).
\]
 According to Theorem \ref{th1} this metric is Ricci flat if
 \[2\varrho  \phi^{\prime\prime}+(2-n) \phi^\prime -4k^2 \phi=0,\]
 and the  initial condition is $\phi(0)=1$. Of course, this equation can be solved using the power series techniques  in Theorem \ref{pptheo} exhibiting the difference between $n$ odd and even. But, for example, when $n=4$ this equation becomes
 \[\varrho  \phi^{\prime\prime}- \phi^\prime -2k^2 \phi=0.\]
The general  solution with $\phi(0)=1$ is given 
  in closed form using the modified Bessel functions $I_2$ and $K_2$
as 
  \[
\phi(\varrho)=  4k^2\varrho K_2(2k\sqrt{2\varrho}) + C \varrho I_2( 2k\sqrt{2\varrho}),
\]
where $C$ is an arbitrary constant.
This solution is not analytic at $\varrho=0$, since the function $\varrho\mapsto \varrho K_2(2k\sqrt{2\varrho})$ fails to have a bounded second derivative at $\varrho=0$. 
The function $\varrho\mapsto
\varrho I_2( 2k\sqrt{2\varrho})
$  is analytic and the constant $C$  is a remnant of the non-uniqueness. In contrast, when $n=5$ the equations becomes
\[\varrho  \phi^{\prime\prime}- \frac{3}{2}\phi^\prime -2 \phi=0.\]
The general  solution with $\phi(0)=1$ is given by
\begin{eqnarray*}
\phi(\varrho)&=& 
\cosh (2k\sqrt{2\varrho}) -2k \sqrt{2}\sinh (2\sqrt{2\varrho}) + \frac{8k^2}{3}\, \varrho\, \cosh (2k\sqrt{2\varrho})
\\
&&+ 3\,C\left(   \sinh (2k\sqrt{2\varrho})
-2k \sqrt{2\varrho}\, \cosh (2k\sqrt{2\varrho})
+ \frac{8k^2}{3}\,\varrho\,\sinh (2k\sqrt{2\varrho}) 
\right),
\end{eqnarray*} 
with an arbitrary constant $C$. The unique analytic solution  is obtained  by $C=0$.
\end{example}
\subsection{On the ambient holonomy for conformal pp-waves}
We conclude our observations about conformal pp-waves by describing the holonomy of their ambient metrics. In this section we describe the possible holonomy for ambient metrics of conformal pp-waves.

\begin{remark}[Terminology]
\label{terminology}
In this and some of the following sections, we will 
determine the holonomy algebras of  semi-Riemannian manifolds $(\widetilde{M}\widetilde{g})$ at some point~$p\in \widetilde{M}$. However, since the holonomy algebras at different points are conjugated  to each other in $\mathbf{SO}(r,s)$, we will not make the point explicit in our notation. For a parallel tensor field $\Phi$, i.e. with $\widetilde{\nabla} \Phi=0$, we will use the  terminology {\em stabilizer of $\Phi$} by which we mean all linear maps $H\in \soa(T_p\widetilde{M},g_p)$   that  act trivially on $\Phi|_p$, i.e. $H\cdot \Phi|_p=0$. If $\widetilde{\nabla} \Phi=0$, then the holonomy algebra is contained in the stabilizer of~$\Phi$.\footnote{Note that the converse is not true: a tensor whose stabilizer contains the holonomy algebra, must not be parallel. 
This can be easily understood in the example of a parallel vector field. If one multiplies the parallel vector field by a function, the resulting vector field is no longer parallel, however  still invariant under the holonomy at each point.}
Moreover, we say that a distribution $\Cal V$ is parallel if it is invariant under parallel transport, or equivalently if $\tnab_X V\in \Gamma(\Cal V)$ for all $X\in T\widetilde {M}$ and all $V\in \Gamma(\Cal V)$. 
The  {\em stabilizer of $\Cal V$} consists of 
 all linear  maps  in $ \soa(T_p\widetilde{M},g_p)$  that leave $\Cal V|_p$ invariant.  Again, if there is a parallel distribution, then the holonomy algebra is contained in its stabilizer.\end{remark}

First we note that a pp-wave  $(M,g)$  as defined in \eqref{pp} admits a parallel null distribution  $\mathcal V$ spanned by the parallel vector fields $\frac{\del}{\del v^a}$, for $a=1,\ldots, p$.  Moreover, the image of the Ricci endomorphism is contained in $\mathcal V$. Then, on the other hand,  the results in \cite[Theorem 1.1]{leistner-nurowski12}, in the case $p=1$, and the generalization in \cite[Theorem 1]{Lischewski14} imply that the normal conformal tractor bundle  admits a totally null subbundle of rank $p+1$ that is parallel for the normal conformal tractor connection of $[g]$, i.e., its fibres are invariant under the conformal holonomy. On the other hand, in the case of odd-dimensional analytic pp-waves, the general theory in \cite{graham-willse11} ensures that parallel objects of the tractor connection carry over to parallel objects of the analytic ambient metric. Hence, for an odd dimensional analytic pp-wave,  the  holonomy of the analytic ambient metric in \eqref{anaambientpp} admits an invariant totally null plane of rank $p+1$. However, in the next theorem we verify and strengthen this result by directly determining the holonomy of general metrics of the form \eqref{ambientpp}. We will prove the following result in the case of $p=1$ and in Lorentzian signature but it clearly generalises to larger $p$ and other signatures in an obvious way.
\begin{theorem}
Let 
 $F=F(\varrho, u,x^1,\ldots, x^{n-2})$  be a smooth function on $\R^{n}$.
Then the metric
\[
\widetilde{g}_F
= 2 \der (t \varrho  ) \der t
 + t^2 \Big( 2\der u(\der v+F\,\der u)+\delta_{AB}\der x^A\, dx^B\Big)
 \]
 on $\widetilde M=\R^{n+1}\times \R_{>0}$
satisfies the following properties:
\begin{enumerate}\setlength{\labelsep=2pt}
\item The distribution $\widetilde{\mathcal V}=\mathrm{span}(\del_v,\del_\varrho)$ is parallel, i.e., it is invariant under parallel transport.
\item Let $\mathcal V=\R\cdot \del_v$ be the distribution of null lines spanned by $\del_v$ and  $\mathcal V^\perp$ be the distribution of vectors in $T\widetilde M$ that are orthogonal to $\del_v$, i.e., $\mathcal V^\perp=\mathrm{span} (\del_v,\del_\varrho,\del_1,\ldots ,\del_{n-2},\del_t)$. Then
the curvature $\widetilde{R} $ of $\widetilde{g}_F$ satisfies
\[\widetilde{R}(U,V)Y=0,\ \ \text{ for all $U,V\in \mathcal V^\perp$ and $Y\in T\widetilde{M}$.} \]
\item The holonomy algebra $\mathfrak{hol}(\widetilde{g}_F)$ of $\widetilde{g}_F$ is contained in 
\[
\Big(\mathfrak{sl}_2\R\ltimes (\R^2\otimes \R^{n-2})\Big)\oplus\R
 =
 \left\{{\small
\begin{pmatrix}
X & Z & a J
\\[1mm]
0&0&-Z^\top
\\[1mm]
0&0 &-X^\top
\end{pmatrix}}
\mid
\begin{array}{l}
 X\in \mathfrak{sl}_2\R, \\ Z\in \R^2\otimes \R^{n-2},\\
  a\in \R\end{array}\right\},
\]
where $J$ is the $2\times 2$-matrix $ J={\small \begin{pmatrix}0&1\\-1&0\end{pmatrix}} $.
\end{enumerate} 
\end{theorem}
\begin{proof}
Let $\widetilde{\nabla}$ be the Levi-Civita connection of $\widetilde{g}_F$. Then a straightforward computation reveals that
\begin{equation}\label{nabpp}
\begin{array}{rcl}
\widetilde{\nabla} \del_v&=& \tfrac{1}{t}\del_v\otimes dt - \del_\varrho\otimes du,
\\[1mm]
\widetilde{\nabla} \del_\varrho&=&\tfrac{1}{t} \del_\varrho\otimes dt
+  F_\varrho\, \del_v\otimes du, \\[1mm]
\widetilde{\nabla} \del_A&=&
\tfrac{1}{t} \del_A\otimes dt
+
F_A\,  \del_v\otimes du -
 \del_\varrho\otimes dx^A,\qquad A=1, \ldots , n-2,
\end{array}\end{equation}
in which $F_A=\del_A(F)$, and $F_\varrho=\del_{\varrho}(F)$. The first two equations show that the distribution $\widetilde{\mathcal V}$ is invariant under parallel transport, and also that, in general, there is no invariant null line in $\widetilde{\mathcal V}$. Moreover, it allows us to show that the curvature satisfies
$\widetilde{R}(X,Y)\del_v=0$ for all $X,Y\in T\widetilde M$, 
\begin{eqnarray*}
\widetilde{R}(\del_A,\del_u)\del_\varrho&=&\widetilde{R}(\del_\varrho,\del_u)\del_i\ =\ F_{A\varrho}\,\del_v
\\
\widetilde{R}(\del_\varrho,\del_u)\del_\varrho&=&F_{\varrho\varrho}\,\del_v,
\\
\widetilde{R}(\del_u,\del_B)\del_A&=&-(\delta_{AB}F_\varrho+ F_{AB})\,\del_v,
\end{eqnarray*}
and that all other terms of the form 
 $\widetilde{R}(X,Y)\del_\varrho$ and $\widetilde{R}(X,Y)\del_i$ are zero, unless the symmetry of $\widetilde{R}$ prevents this. This shows that the image of $\widetilde{\mathcal V}^\perp$ under $\widetilde R (X,Y)$  is contained in $\R\cdot \del_v$. The symmetries of the curvature then imply part (ii) of the Theorem.
 
For the last part, first we note that, since $\widetilde{\mathcal V}$ is parallel, the holonomy algebra of $\widetilde g$ is contained in the stabilizer in $\mathfrak{so}(2,n)$ of the  totally null plane $\widetilde{\mathcal V}$, which is equal to 
\begin{equation}\label{holgl}
(\mathfrak{gl}_2\R\oplus \mathfrak{so}(n-2))\ltimes ( \R^2\otimes \R^{n-2}\oplus \R )
=
\left\{{\small
\begin{pmatrix}
X & Z & a J
\\[1mm]
0&S&-Z^\top
\\[1mm]
0&0 &-X^\top
\end{pmatrix}}
\mid
\begin{array}{l}
 X\in \mathfrak{gl}_2\R,\   a\in \R \\ Z\in \R^2\otimes \R^{n-2},\\
 S\in \mathfrak{so}(n-2)\end{array}\right\}.
\end{equation}
Moreover, the equations \eqref{nabpp} show that the $2$-form
$\mu=t dt\wedge du$ is parallel with respect to the Levi-Civita connection. Hence, the projection of the holonomy to $\mathfrak{gl}_2\R$ actually lies in $\mathfrak{sl}_2\R$, i.e., $X\in\mathfrak{sl}_2\R$. Note that $\R \cdot J$ commutes with all of 
$(\mathfrak{sl}_2\R\oplus \mathfrak{so}(n-2))\ltimes ( \R^2\otimes \R^{n-2})$, so the holonomy reduces to
\[\Big(\mathfrak{sl}_2\R\oplus \mathfrak{so}(n-2))\ltimes ( \R^2\otimes \R^{n-2})\Big)\oplus \R .\]
Finally, 
to show that the elements in the  holonomy algebra have no   $\mathfrak{so}(n-2)$-component, i.e., that $S=0$ in \eqref{holgl}, 
  we fix a point $p\in \widetilde{M}$ and use the Ambrose-Singer Holonomy Theorem to show that $\mathfrak{hol}_p(\widetilde{g})$ maps $\widetilde{\mathcal V}^\perp|_p$ to $\widetilde{\mathcal V}|_p$. Indeed, let 
$V\in \widetilde{\mathcal V}^\perp|_p$ be in the fibre of $\widetilde{\mathcal V}^\perp$ at $p$. Let $V_q=\mathcal P_\gamma(V)$ denote  the parallel transport  of $V$ along a curve that ends at $q\in \widetilde{M}$. By the invariance of the distribution $\widetilde{\mathcal V}^\perp$ we have  $V_q\in \widetilde{\mathcal V}^\perp|_q$, however applying 
the curvature at $q$, $\widetilde R_q(X,Y)$ for $X,Y\in T_q\widetilde M$, gives 
\[ \widetilde R_q(X,Y)V_q\ \in\  \mathcal V|_q\ \subset\ \widetilde{\mathcal V}|_q,\]
by the above formulae.
Since also $\widetilde{\mathcal V}$ is invariant under parallel transport, we obtain
\[\mathcal P^{-1}_\gamma\circ R(X,Y) \circ \mathcal P_\gamma (V)\in \widetilde{\mathcal V}|_p,\]
for every $V\in \widetilde{\mathcal V}^\perp|_p$.
Then  the Ambrose-Singer  Theorem provides us with the third statement.
\end{proof}
For a conformal pp-wave  defined by a function $H$, this theorem for $F=H+h$
gives an upper bound for the  holonomy algebra of the 
 the analytic ambient metric $\widetilde g$ in \eqref{anaambientpp} in the cases when it exists, e.g., when $n$ odd or when $\Delta^{\frac{n}{2}}H=0$. This upper bound  
 is an improvement of the result in \cite[Corollary 2]{leistner-nurowski08}. 
  The special structure of the ambient metric, in particular its Ricci flatness, might reduce the holonomy further. Indeed, for special  conformal pp-waves in Lorentzian signature, 
 such as plane waves or   Cahen-Wallach spaces, the ambient holonomy   reduces further: 
 \begin{proposition}
 Let $(S_{AB})_{i,j=1}^{n-2}$ be a symmetric  matrix of functions $S_{AB}=S_{AB}(u)$ with trace $S=S(u )$
 and let 
 \begin{equation}
 \label{plane-wave}
 g=2\der u\,\der v+2\left(S_{AB}(u)\, x^Ax^B\right) \der u^2+\delta_{AB}\der x^A\, dx^B
 \end{equation}
 be the corresponding Lorentzian plane wave metric on $\R^n$.
If $f=f(u)$ is  a solution to the ODE
 \[f''=(f')^2 -\tfrac{2}{n-2}S,\]
 then  the metric $\hat g=\mathrm{e}^{2f}g$ is Ricci-flat. Moreover, the metric
  \[\widetilde{g}= 2d(t\varrho )dt +t^2\mathrm{e}^{2f(u)} g\]
  is an ambient metric for the conformal class $[g]$
 and admits two parallel null vector fields $\frac{1}{t}\del_\varrho$ and 
 $\frac{1}{t}(\del_v+h\,\del_\varrho)$, where $h=h(u)$ is a solution to $h'=\mathrm{e}^{2f}$. Consequently, the holonomy algebra of $\widetilde{g}$ is contained in $ \R^2\otimes \R^{n-2}$.
 \end{proposition}
 This proposition is a generalisation of the corresponding statements 
 in \cite{leistner05a} about the conformal holonomy of plane waves. Its proof follows from straightforward computations and the well-known formula \eqref{einstein-ambient} for the ambient metric of Einstein metrics in the case of $\Lambda=0$. A class of examples to which this proposition applies are the symmetric Cahen-Wallach spaces for which the metric \eqref{plane-wave} is defined by a {\em constant} symmetric matrix $S_{AB}$.
\section{Generic 2-distributions in dimension 5 and their ambient metrics}
\label{secg2}
\subsection{Conformal structures and generic rank $2$ distributions in dimension $5$}

The examples of 5-dimensional conformal structures considered in Ref. \cite{leistner-nurowski09} belong to a wider class of conformal metrics naturally associated with \emph{generic rank 2 distributions in dimension 5}. The correspondence between these distributions and  conformal structures with signature $ (2,3)$ is explained in detail in \cite{nurowski04}. Here we recall this correspondence briefly for a particular subclass of these distributions:

Associated with a differential equation $$z'=F(x,y,p,q,z),$$
for real functions $y=y(x)$ and $z=z(x)$, where $p=y'$, 
$q=y''$,  there is a 5-manifold $M$ parametrized by $(x,y,p,q,z)$, and a distribution 
\be
{\mathcal D}={\rm Span}\Big(\partial_q,\partial_x+p\,\partial_y+q\,\partial_p+F\,\partial_z\Big).\label{F1}\ee
The distribution is {\em generic} if  $F_{qq}\neq 0$. For generic distributions the fundamental differential invariants of Cartan \cite{cartan10} are in one-to-one correspondence with \emph{conformal} invariants of a certain conformal class $[g_{\mathcal D}]$ of metrics of signature $(3,2)$ on $M$. In this section we will consider distributions and the corresponding conformal classes defined by 
\be\label{F}
F=q^2+f(x,y,p)+h(x,y)z,\ee where  $f$ and $h$ are smooth functions of their variables. For such $F$ we denote the distribution $\mathcal D$ by $\mathcal D_{f,h}$. The conformal class $[g_{\mathcal D_{f,h}}]$ may be represented by a metric 
\begin{eqnarray*}
g_{{\mathcal D}_{f,h}}
&=&
2\theta^1\theta^5-2\theta^2\theta^4+(\theta^3)^2,
\end{eqnarray*}
where the co-frame $\theta^i$ is given by
\begin{equation}
\label{omegas}
\begin{array}{rcl}
\theta^1&=&\der y-p\der x,\\
 \theta^2&=& \der z-(q^2+f+h z)\der x-\tfrac{\sqrt{2}}{2}q\theta^3,\\
 \theta^3&=&2\sqrt{2}(\der p-q\der x),\\
\theta^4&=&3\der x,\\
\theta^5&=&\tfrac{\sqrt{2}h}{2}\theta^3-6\der q+3(2h q+f_p)\der x+\tfrac{1}{10}\left(9f_{pp}+4h ^2-6(ph_y+h_x)\right)\theta^1,
\end{array}
\end{equation}
and where $f_p$ denotes the partial derivative $\del_p(f)$, $h_x=\del_x(h)$, etc.

In the examples we gave in \cite{leistner-nurowski09}, the function $f$ was just a polynomial of degree $6$ in $p$, and $h\equiv b$ was a constant. We were able to show that for a generic choice of $b$ and a generic choice of the coefficients of the polynomial defining $f$, the ambient metric  had the full exceptional Lie group  ${\bf G}_2$ as holonomy, thus obtaining an 8-parameter family of $\mathbf{G}_2$-metrics.
This   stimulated further research in \cite{graham-willse11}, where the authors showed that the conformal structures $[g_{\mathcal D}]$ associated with analytic generic rank $2$-distributions in dimension 5,  the  holonomy of the corresponding analytic ambient metric is {\em always}  contained in ${\bf G}_2$ and  \emph{generically}  equal to  ${\bf G}_2$.
Motivated by this result,
we  now produce more explicit examples of ambient metrics for the conformal structures associated with distributions ${\Cal D}_{f,h}$ with $f$ and $ h$ defining $F$ as in (\ref{F}). As a byproduct we will obtain a large class of explicit metrics with  holonomy equal to   $\mathbf{G}_2$, extending significantly our examples in \cite{leistner-nurowski09}. 
\subsection{An ansatz for the ambient metric, analytic solutions, and $\mathbf{G}_2$-ambient metrics}
\label{g2sec}
In
this section  we will find explicit formulae for all the Fefferman-Graham ambient metrics for  conformal classes $(M,[g_{{\mathcal D}_{f,h}}])$, with distribution ${\mathcal D}_{f,h}$ related to $F=q^2+f(x,y,p)+h(x,y)z$ via  
\eqref{F1}.
Since conformal structures $[g_{{\mathcal D}_{f,h}}]$ associated with distributions ${\mathcal D}_{f,h}$ are  defined in  \emph{odd} dimension $n=5$, 
the result by Fefferman-Graham implies that  there is a unique analytic ambient metric $\tilde{g}_{{\mathcal D}_{f,h}}$ for each pair of \emph{analytic} functions $f$ and $h$. In order to find this ambient metric,   
we start with the following observation which parallels what we   found previously
for pp-waves:
A direct computation shows that 
the Schouten tensor for the class $[g_{{\mathcal D}_{f,h}}]$ has the form
\[P=\alpha\,(\theta^1)^2+2\beta\,\theta^1\theta^4+\gamma\, (\theta^4)^2,\] with $\theta^1$ and $\theta^4$ as in \eqref{omegas}, and the functions $\alpha$, $\beta$ and $\gamma$ determined by $f$ and $h$ and their 
derivatives:
\begin{equation}
\label{schouten}
\alpha=-\tfrac{3}{80}f_{pppp},\quad 
\beta=-\tfrac{1}{80}(f_{ppp}+6h_y)
,\quad 
\gamma= -\tfrac{1}{15}\left(\tfrac{1}{2}(h_x+ph_y) +\tfrac{1}{12}f_{pp}-\tfrac{1}{3}	h^2\right)
\end{equation}
Hence, as for pp-waves, we make an \emph{ansatz} for the ambient metric $\tilde{g}_{{\mathcal D}_{f,h}}$ in which $g_{{\mathcal D}_{f,h}}(x^i,\varrho)$ assumes a similar form to  $P$. Explicitly, we make the following ansatz for $\tilde{g}_{{\mathcal D}_{f,h}}$:
\be
\tilde{g}_{{\mathcal D}_{f,h}}\ =\ 2\der t\der(\varrho t)+t^2g_{{\mathcal D}_{f,h}}(x^i,\varrho)\ =\ 
2\der t\der(\varrho t)+t^2\big(g_{{\mathcal D}_{f,h}}+A\, (\theta^1)^2+2B\, \theta^1\theta^4+C\,(\theta^4)^2\big),
\label{amf}
\ee
with \emph{unknown} functions $A=A(x,y,p,\varrho)$, $B=B(x,y,p,\varrho)$ and $C=C(x,y,p,\varrho)$.
Miraculously, for this ansatz,  the equations for $Ric( \tilde{g}_{{\mathcal D}_{f,h}})=0$  form a system of PDEs which  are {\em linear} in the unknowns $A,B,C$:
\begin{theorem}\label{linsys1}
The metric $\tilde{g}_{{\mathcal D}_{f,h}}$ in \eqref{amf} is an ambient metric for the conformal class $(M,[g_{{\mathcal D}_{f,h}}])$ if and only if the unknown functions $A=A(x,y,p,\varrho)$, $B=B(x,y,p,\varrho)$ and $C=C(x,y,p,\varrho)$ satisfy the initial conditions
$A_{|\varrho=0}\equiv 0$, $ B_{|\varrho=0}\equiv 0$, $ C_{|\varrho=0}\equiv 0$, and the following system of PDEs:
\be
\begin{array}{rcl}
LA&=&\tfrac{9}{40}f_{pppp}\\[2mm]
LB&=&-\tfrac{1}{36}A_p+\tfrac{3}{40}f_{ppp}+\tfrac{9}{20}h_y\\[2mm]
LC&=&-\tfrac{1}{18}B_p+\tfrac{1}{324}A+\tfrac{1}{30}f_{pp}-\tfrac{2}{15}h^2+\tfrac{1}{5}( ph_y+h_x),
\end{array} \label{2sys}
\ee
where the \emph{linear} operator $L$ is given by
$$
L=2\varrho\frac{\partial^2}{\partial\varrho^2}-3\frac{\partial}{\partial\varrho}-\tfrac{1}{8}\frac{\partial^2}{\partial p^2}.
$$ 
\end{theorem}
The proof of this theorem 
proceeds by  directly computing the Ricci-tensor 
$Ric( \tilde{g}_{{\mathcal D}_{f,h}})$, using the formulae \eqref{schouten} for 
 the Schouten tensor of $g_{{\mathcal D}_{f,h}}$.

In oder  to find general formal solutions for the linear system \eqref{1sys} we assume the power series expansions of the unknowns
$A=A(x,y,p,\varrho)$, $B=B(x,y,p,\varrho)$ and $C=C(x,y,p,\varrho)$, i.e,
$$A=\sum_{k=1}^\infty a_k(x,y,p)\varrho^k,\quad B=\sum_{k=1}^\infty b_k(x,y,p)\varrho^k,\quad C=\sum_{k=1}^\infty c_k(x,y,p)\varrho^k.$$
We describe the solutions in the following Theorem. The proof that these give all analytic solutions is analogous to the pp-wave case.

\begin{theorem}\label{as}
The most general solution to the linear system (\ref{2sys}), vanishing at $\varrho=0$ and \trev{being analytic in $\varrho$ whenever $f$ is analytic,}  is given by
 \beqn
A&=&\frac{3}{5} \sum_{k=1}^\infty\frac{(2k-1)(2k-3)}{2^{2k}(2k)!}\,\frac{\partial^{(2k+2)}f}{\partial p^{(2k+2)}}\, \varrho^k,\\
B&=&
-\frac{3}{20}\varrho h_y
-\frac{1}{15}\sum_{k=1}^\infty\frac{(2k-1)(2k-3)(2k-5)}{2^{2k}(2k)!}\,\frac{\partial^{(2k+1)}f}{\partial p^{(2k+1)}}\, \varrho^k,\\
C&=&
\frac{2}{45}\varrho h^2
-\frac{1}{15}\varrho (ph_y+h_x)
+
\frac{2}{135} \sum_{k=1}^\infty\Big(\frac{(k-{ 3})(2k-1)(2k-3)(2k-5)}{2^{2k}(2k)!}\,\frac{\partial^{2k}f}{\partial p^{2k}}\Big)\, \varrho^k.
\eeqn
\end{theorem}
Our next goal is to  determine explicitly the parallel $3$-form that gives the  reduction to  $\mathbf{G}_2$ of the  holonomy of the ambient metric in \eqref{amf}. 
Such a  holonomy reduction imposes conditions on the connection and hence implies first order conditions on the metric coefficients.
Therefore, in order to find the $3$-form, we have to identify these first order conditions. 
Our approach will be as follows: 
From \cite{graham-willse11} we know that for analytic functions  $f$ and $h$
 the \trev{unique analytic} ambient metric will admit a parallel $3$-form. This means that the above analytic solutions  to the second order system \eqref{2sys}, which give us  the unique analytic Ricci flat ambient metric, will in turn, via the  induced holonomy reduction,  be the solutions to some first order equations on the functions $A$, $B$ and $C$. It turns out that the first order system we will find in this way  actually  implies the second order  equations \eqref{2sys} as their integrability conditions.

\begin{theorem}\label{1systheo}
Let $f=f(x,y,p)$ and $h=h(x,y)$ be two smooth functions. Then the 
 \trev{formal} solutions $A$, $B$ and $C$ 
of the second order system  \eqref{2sys}  
solve  
 the following first order system
\begin{equation}\label{1sys}
\begin{array}{rcl}
B_p
&=& \frac{5}{9}A-\frac{2}{9}\varrho A_\varrho 
\\[2mm]
B_\varrho &=& -\frac{1}{72}A_p - \frac{1}{40}f_{ppp} -\frac{3}{20}h_y
\\[2mm]
C_\varrho &=& \frac{1}{648} A - \frac{1}{72}B_p - \frac{1}{90} f_{pp}
+ \frac{2}{45}h^2 -\frac{1}{15}(ph_y+h_x)
\\[2mm]
C_p &=&  \frac{1}{324}\varrho A_p  +  \frac{2}{3}B + \frac{1}{180}\varrho f_{ppp}
+\frac{1}{30}\varrho h_y.
\end{array}
\end{equation}
Moreover, 
\trev{twice  differentiable  solutions}  $A$, $B$ and $C$ of this  first  order system \eqref{1sys} are  solutions to the   second order system \eqref{2sys}.
\end{theorem}
\begin{proof}
It is a matter of checking that the \trev{formal} solutions in Theorem \ref{as} of  the second order system \eqref{2sys}  satisfy the first order system \eqref{1sys}.

In order to derive the second order equations from the first order system,  we first use the integrability condition arising from the first two equations in  \eqref{1sys}. We immediately compute
\[
0=
B_{\varrho p}-B_{p \varrho}
=
\tfrac{1}{9}\left( L(A) -\tfrac{9}{40}f_{pppp}\right),
\]
i.e., the first equation in \eqref{2sys} is exactly this  integrability condition. 
Note that the integrability condition arising from  the last two equations, $C_{p \varrho}- C_{\varrho p}=0$, immediately {\em follows}  from the first two equations in \eqref{1sys}:
\begin{align*}
C_{p \varrho}- C_{\varrho p}\
&=\ 
\tfrac{1}{648} A_p + \tfrac{1}{324}\varrho A_{p\varrho} +\tfrac{1}{60}f_{ppp}+\tfrac{1}{10}h_y
+\tfrac{1}{72}B_{pp} +\tfrac{2}{3}B_\varrho
\ =\ 0,
\end{align*}
when substituting $B_{pp}=\tfrac{5}{9}A_p-\tfrac{2}{9}\varrho A_{\varrho p}$ and 
 $B_\varrho=-\tfrac{1}{72}A_p -\tfrac{1}{40}f_{ppp} -\tfrac{3}{20}h_y$ in the last two terms.
As the first derivatives of $B$ and $C$ in the variables $p$ and $\varrho$ are explicitly given in 
\eqref{1sys}, one checks that
the second equation in \eqref{2sys} can be derived directly from the first two equations in \eqref{1sys}, whereas the third equation in \eqref{2sys} requires all four equations in \eqref{1sys} for its derivation.
\end{proof}
\begin{remark}\label{na-rem}
In Section \ref{nona-sec}, when  determining {\em all} \trev{formal} solutions to the second order system \eqref{2sys}, we will see that \eqref{2sys} has  solutions, \trev{defined on $\varrho\ge 0$ and only twice differentiable} which do not solve the first order system \eqref{1sys}, see Remark \ref{1sysrem}.
\end{remark}
\begin{remark} 
Note that for a set of  solutions $A$, $B$, $C$ of \eqref{1sys}, the functions $B$ and $C$ are uniquely determined by the function $A$ and the initial conditions.

\end{remark}

The first order system (\ref{1sys}) enables us to find the parallel $3$-form that defines the $\mathbf{G}_2$-reduction of the ambient metric.
To this end we define 
a null co-frame for the ambient metric \eqref{amf}
by
\begin{equation}
\label{g2nullframe}
\begin{array}{l}
\omega^1 
= 9 \sqrt{2} dt + t\sqrt{2} h\,\theta^4,
\qquad
\omega^2  
 = \theta^1, 
\qquad
\omega^3
 = -\frac{1}{9}t\, h( \varrho dt + t d\varrho)  -t^2 \theta^2 + \frac{1}{2}\,t^2\,C\, \theta^4,
 \\[2mm]
\omega^4 
= t\theta^3,
\qquad
\omega^5 
= \theta^4,
\qquad
\omega^6 
= \frac{1}{2}\:t\,A \, \theta^1 + t^2 B\, \theta^4 +t^2 \theta^5,
\qquad
\omega^7 
=  \frac{\sqrt{2}}{18}( \varrho dt +   t d\varrho  ),
\end{array}
\end{equation}
with the $\theta^i$'s given in the equations \eqref{omegas}. We obtain:
\begin{theorem}
Let $\mathcal D_{f,h}$ be a generic distribution on $\R^5$ defined by 
\[
F=q^2+f(x,y,p)+h(x,y)z,\] where  $f$ and $h$ are smooth functions of their variables, and let  $\left[g_{{\mathcal D}_{f,h}}\right]$ be the corresponding conformal class.  Let 
$A=A(x,y,p,\varrho)$, $B=B(x,y,p,\varrho)$ and $C=C(x,y,p,\varrho)$ be twice differentiable solutions of the  first order system \eqref{1sys} associated to $f$ and $h$, with initial conditions $A|_{\varrho=0} =B|_{\varrho=0}=C|_{\varrho=0}=0$. Then, in the frame defined in \eqref{g2nullframe}, an ambient metric 
is given as
\begin{equation}\label{FG-omega}
\tilde{g}_{{\mathcal D}_{f,h}}=2\omega^1\omega^7 +2\omega^2\omega^6+2\omega^3\omega^5 + (\omega^4)^2,\end{equation}
with a parallel 
 $3$-form 
\begin{equation}\label{3form-omega}
\Upsilon = 2 \omega^{123}- \omega^{147} -
\omega^{246}  - \omega^{345} + \omega^{567},
\end{equation}
where we use the standard notation $\omega^{ijk}=\omega^i\wedge\omega^j\wedge\omega^k$. 

Moreover, if \trev{$f$ is analytic and we take $A$, $B$ and $C$ to be the solutions 
 of the system \eqref{1sys} given in Theorem \ref{as} and being
 analytic in $\varrho$,   then the metric in \eqref{FG-omega} is the unique  ambient metric that is analytic in $\varrho$.}
\end{theorem}
\begin{proof}
The Ricci-flatness  of $\widetilde{g}_{D_{f,h}}$ as in \eqref{amf} follows from Theorems \ref{linsys1} and \ref{1systheo}. 
The covariant derivative $\widetilde{\nabla} \Upsilon$ with respect to the Levi-Civita connection $\widetilde{\nabla} $ of 
$\tilde{g}_{{\mathcal D}_{f,h}}$ in \eqref{amf} contains  derivatives of $f$ and $h$ as well as first derivatives 
of $A$, $B$ and $C$. Computing $\widetilde{\nabla} \Upsilon$ and substituting the first derivatives 
of $A$, $B$ and $C$ with the help of the first order system \eqref{1sys} shows that $\widetilde{\nabla} \Upsilon=0$.
\end{proof}
\trev{
\begin{remark}
Note that this theorem not only gives us the unique analytic solution to the Fefferman-Graham equations but, in principle, also applies also to  possible non-smooth solutions to the first order system \eqref{1sys}. More generally, leaving the context of conformal structures aside, it provides a tool for constructing  metrics with holonomy contained in $\mathbf{G}_2$ from solutions of the first order system  \eqref{1sys}. Below in Theorem \ref{suffcond} we will see what one has to assume for the solutions  in order to ensure that such metrics have holonomy equal to $\mathbf{G}_2$.
\end{remark}}
For completeness, we recall that there is a one-to-one correspondence between a 3-form making the reduction from $\sog(4,3)$ to ${\bf G}_2$ and a spinor. Using the formula governing this correspondence (see e.g. \cite{Kath98g2} or \cite[pp. 430]{leistner-nurowski09}) we now find the spinor $\psi$ corresponding to the 3-form $\Upsilon$. This needs some preparations:

Since the spinor $\psi$ takes the most convenient form in an orthonormal frame, we pass from the null co-frame $(\omega^i)$ to a co-frame $\xi^i$ given by
\begin{alignat*}{4}
\xi^1 &= \omega^1 + \frac12 \omega^7, 
&\quad
\xi^2 &=\frac{\sqrt{2}}{2} \omega^2 +  \frac{\sqrt{2}}{2} \omega^6,
&\quad
\xi^3 & = \frac{\sqrt{2}}{2} \omega^3 +  \frac{\sqrt{2}}{2} \omega^5,
&\quad
\xi^4 &=  \omega^4,
\\
\xi^5 & = \frac{\sqrt{2}}{2} \omega^3 -  \frac{\sqrt{2}}{2} \omega^5,
&\quad
\xi^6 & = \frac{\sqrt{2}}{2} \omega^2 -  \frac{\sqrt{2}}{2} \omega^6,
&\quad
\xi^7 &= \omega^1 - \frac12 \omega^7.
\end{alignat*}
This co-frame is orthonormal for $\widetilde{g}_{f,h}$,
\[\widetilde{g}_{f,h}\ =  \widetilde{g}_{ij} \xi^i\xi^j\ =\  (\xi^1)^2 +(\xi^2)^2+(\xi^3)^2 +(\xi^4)^2 -(\xi^5)^2 -(\xi^6)^2-(\xi^7)^2,
\]
and the parallel 3-form \eqref{3form-omega} becomes\begin{equation*}
\Upsilon = \xi^{125} - \xi^{136} + \xi^{147} +\xi^{237} +\xi^{246} +\xi^{345} - \xi^{567}.
\end{equation*}
To find  the formula for the spinor $\psi$, we need a representation of the Clifford algebra $\mathrm{Cl}(4,3)$. To this end we introduce real $8\times 8$  matrices $\s_i$, called the $\s$-matrices, satisfying the relation
\[
\s_i\s_j + \s_j\s_i = 2 \widetilde{g}_{ij}\mathbb{I}_8.\]
They generate the Clifford algebra $\mathrm{Cl}(4,3)$ acting on the vector space $\bbR^8$ of real spinors. Following \cite[pp. 430]{leistner-nurowski09} we write the $\s$-matrices as\footnote{In fact, the $\s$-matrices come, as a special case,  from A.~Trautman's inductive construction of real representations for $\mathrm{Cl}(g)$ in case of  metrics $g$ with split signature. This construction can be found in 
 \cite{Trautman06}.}
\[\begin{array}{cccc}
\s_{1}= \left(\begin{array}{cc}
0&\gamma_1 \\ \gamma_1 & 0
\end{array}\right),&
\s_{2}= \left(\begin{array}{cc}
0&\gamma_3\\\gamma_3&0
\end{array}\right),
&
\s_3= \left(\begin{array}{cc}
0&\gamma_5\\\gamma_5&0
\end{array}\right)
&
\s_4= \left(\begin{array}{cc}
\1_4&0\\0&-\1_4
\end{array}\right)
\\[4mm]
\s_5= \left(\begin{array}{cc}
0&\gamma_2 \\ \gamma_2 & 0
\end{array}\right),&
\s_6= \left(\begin{array}{cc}
0&\gamma_4\\\gamma_4&0
\end{array}\right),
&
\s_7= \left(\begin{array}{cc}
0&-\1_4\\\1_4&0
\end{array}\right).
&
\end{array}\]
where
\[
\begin{array}{lll}
\gamma_{1}=
{\scriptsize \left(\begin{array}{rrrr}
0&0&0&1\\
0&0&1&0\\
0&1&0&0\\
1&0&0&0
\end{array}\right)},&
\gamma_2={\scriptsize
 \left(\begin{array}{rrrr}
0&0&0&-1\\
0&0&1&0\\
0&-1&0&0\\
1&0&0&0
\end{array}\right)},
&
\gamma_3=
{\scriptsize
 \left(\begin{array}{rrrr}
0&0&1&0\\
0&0&0&-1\\
1&0&0&0\\
0&-1&0&0\\
\end{array}\right)},
\\
\\
\gamma_4= \left(\begin{array}{lr}
0&-\1_2\\
\1_2&0\\
\end{array}\right)
,&
\gamma_5=  \left(\begin{array}{lr}
\1_2&0\\
0&-\1_2
\end{array}\right).
\end{array}\]
We identify the spinor bundle $\mathcal S\to\tilde{M}$ with a vector bundle over $\tilde{M}$, with fibre $\R^8$ on which the orthonormal basis $e_i$, dual to $\xi^i$ (i.e. $ \xi^j(e_i)=\delta_i^{~j}$), acts via the \emph{Clifford multiplication}, i.e., via the multiplication defined by 
$$e_i\psi=\s_i\psi.$$
Then the lift of the Levi-Civita connection from $\tilde{M}$ to the spinor bundle $\mathcal S$ is given by
\[\widetilde{\nabla} \psi =d \psi+\frac{1}{4}\, \widetilde{\Gamma}^{kl}\, \s_{k} \, \s_{l}\, \psi,
\]
where $\widetilde{\Gamma}^{kl}$ are the Levi-Civita connection $1$-forms for the ambient metric $\widetilde{g}_{f,h}$ in the orthonormal co-frame $\xi^i$. Explicitly they are defined by
$$\widetilde{\Gamma}^{kl}(X)=\widetilde{g}(\widetilde{\nabla}_X \xi^k,\xi^l).$$
With this notation,
the spinor field $\psi$ corresponding to the parallel 3-form $\Upsilon$
 is as simple as 
\[\psi=
\left(
0,
1 ,
 -1, 
0 , 
1,
0,
0,
-1
\right)^\top.
\]
One can easily check that it is parallel, 
$\tilde{\nabla}\psi=0,$
with respect to the connection $\tilde{\nabla}$, \trev{and hence confirms that the metric $\widetilde{g}_{f,h}$ has holonomy in ${\bf G}_2$.}

\bigskip

Next, we address the important question as to when   the conformal classes $[g_{{\mathcal D}_{f,h}}]$ determines an ambient metric with holonomy  {\em equal to} $\mathbf{G}_2$.
 The next theorem gives a very simple sufficient condition.

\begin{theorem}\label{suffcond}
Let  $A$, $B$ and $C$ be solutions to the first order system \eqref{1sys} for functions $f=f(x,y,p)$ and $h=h(x,y)$, and assume that   $A_\varrho\not=0$. Then 
the holonomy   of the ambient metric 
\[ \tilde{g}_{{\mathcal D}_{f,h}}\ =\ 2\der t\der(\varrho t)+t^2g_{{\mathcal D}_{f,h}}(x^i,\varrho)\ =\ 
2\der t\der(\varrho t)+t^2\big(g_{{\mathcal D}_{f,h}}+A\, (\theta^1)^2+2B\, \theta^1\theta^4+C\,(\theta^4)^2\big),\]
for the conformal class defined  as in \eqref{F} by $f$ and $h$ has holonomy {\em equal} to $\mathbf{G}_2$.
\end{theorem}
\begin{proof}

Let $X_i$, $i = 1,  \ldots, 7$ be the vector fields dual to the 1-forms in \eqref{g2nullframe}.
The algebra of derivations,  represented by $(1,1)$-tensors, which preserve the 3-form \eqref{3form-omega} is given by
\begin{alignat*}{2}
h_1 &=X_3\otimes  \omega_2-X_6\otimes  \omega_5, 
\\
h_2 &= X_3\otimes  \omega_1-X_7\otimes  \omega_5,
\\
h _3 &= X_3\otimes  \omega_4-\,X_4\otimes  \omega_5+\,X_6\otimes  \omega_1-\,X_7\otimes  \omega_2,
\\
h_4 &=\,X_2\otimes  \omega_5-\,X_3\otimes  \omega_6+2\,X_4\otimes  \omega_1+-2\,X_7\otimes  \omega_4,\quad
\\
h_5 &= \,X_1\otimes  \omega_5-\,X_3\otimes  \omega_7-2\,X_4\otimes  \omega_2+2\,X_6\otimes  \omega_4,
\\
h_6 &= \,X_2\otimes  \omega_1-\,X_7\otimes  \omega_6,
\\
h _7 &=\,X_1\otimes  \omega_1 -2\,X_2\otimes  \omega_2+\,X_3\otimes  \omega_3-\,X_5\otimes  \omega_5+2\,X_6\otimes  \omega_6- \,X_7\otimes  \omega_7,   \kern -30pt &&
 \\
h_8 & =\,X_2\otimes  \omega_2-\,X_3\otimes  \omega_3+\,X_5\otimes  \omega_5-\,X_6\otimes  \omega_6\\
h_9&= \,X_1\otimes  \omega_2-\,X_6\otimes  \omega_7,\\
h_{10} &= \,X_2\otimes  \omega_4-\,X_4\otimes  \omega_6-\,X_5\otimes  \omega_1+\,X_7\otimes  \omega_3,\\
h_{11} & =\,X_1\otimes  \omega_4-\,X_4\otimes  \omega_7+\,X_5\otimes  \omega_2-\,X_6\otimes  \omega_3,\\
h_{12} &= \,X_1\otimes  \omega_6-\,X_2\otimes  \omega_7+2\,X_4\otimes  \omega_3-2\,X_5\otimes  \omega_4,\\
h_{13} & = \,X_2\otimes  \omega_3-\,X_5\otimes  \omega_6,\\
h_{14} &= \,X_1\otimes  \omega_3-\,X_5\otimes  \omega_7.
\end{alignat*}
Because the 3-form  $\Upsilon$ is parallel, 
the curvature tensor $R(X, Y)$ and its directional covariant derivatives of the curvature tensor lie in the 
span of these $(1,1)$-tensors.
  To prove the theorem, we have to check, conversely, that all the tensors $h_i$ are in the span of the curvature tensor $R(X, Y)$ and its directional covariant derivatives. 
To this end set 
 $\CalH_i = \text{span}\{h_1, h_2, \ldots , h_i\}$.  Then one easily calculates for the curvature and its derivatives
\begin{alignat*}{2}
R(X_3, X_4) &= \frac12 A_{\varrho}\, h_1, 
&\quad
 \nabla_{X_2} h_1 &= -\frac{\sqrt{2}}{18t} \,h_2,
\\ 
\nabla_{X_5} h_1 & = \frac{\sqrt{2}}{18t}\, h_3  \mod  \CalH_2,
&\quad
\nabla_{X_5}h_3 &=  \frac{\sqrt{2}}{18t}\, h_4 \mod \CalH_3, 
\\
\nabla_{X_2}h_4 &= -9t\frac{\sqrt{2}}2 A_{\varrho}\, h_5  \mod \CalH_4,
&\quad
\nabla_{X_5}h_4 &= -\frac{\sqrt{2}}{6t}\, h_6  \mod  \CalH_5,
\\
\nabla_{X_5}h_5 &=  -\frac{\sqrt{2}}{18t}\, h_7  \mod \CalH_6,
&\quad
\nabla_{X_2}h_6 &=  -9t\frac{\sqrt{2}}2 A_{\varrho}\, h_8  \mod\ \CalH_7,
\\
\nabla_{X_2}h_7 &=  27t\frac{\sqrt{2}}2 A_{\varrho}\, h_9  \mod \CalH_8,
&\quad
\nabla_{X_5}h_8 &=  \frac{\sqrt{2}}{18t}\,h_{10}  \mod \CalH_9,
\\
\nabla_{X_2}h_9 &=  \frac{\sqrt{2}}{18t}\, h_{11}  \mod \CalH_{10},
&\quad
\nabla_{X_5}h_{11} &=- \frac{\sqrt{2}}{18t}\ h_{12}  \mod \CalH_{11},
\\
\nabla_{X_5}h_{12} &=  -\frac{\sqrt{2}}{6t}\, h_{13}  \mod \CalH_{12},
&\quad
\nabla_{X_2}h_{13} &=  -9t\frac{\sqrt{2}}2 A_{\varrho}\, h_{14}  \mod  \CalH_{13}.
\end{alignat*}
This establishes  the theorem.
\end{proof}
Let us compare  this sufficient criterion with the one  given in 
\cite{leistner-nurowski09} for conformal classes defined by
 $h(x,y)\equiv b$ and $f(x,y,p)=\sum_{i=0}^6 a_ip^i$ with constants $a_i$ and $b$: There it was shown that such a conformal class has ambient holonomy equal to $\mathbf{G}_2$, if $a_3^2+ a_4^2+a_4^2+a_5^2+a_6^2\not=0$. Clearly the case $f(x,y,p)=p^3 $ is not covered by the new criterion in Theorem \ref{suffcond}, since in this case we have $A\equiv 0$, but as soon as $a_4\not=0$ the new criterion applies.

\subsection{Examples}
Returning to Theorem 3.2,   we see that if we want to  have examples of ambient metrics that \trev{ are {\em polynomial in $\varrho$}  of order at most $k$} we need to have  $\frac{\partial^{(2k+2)}f}{\partial p^{(2k+2)}}\equiv 0$, i.e. the function $f=f(x,y,p)$ defining the distribution must be a polynomial in $p$ of order no higher than $2k+1$.
On the other hand, as soon as we have $\frac{\partial^4 f}{\partial p^4}\not=0$, and hence $A_\varrho\neq 0$, Theorem \ref{suffcond} ensures that the holonomy of the ambient metric is equal to $\mathbf{G}_2$. Hence, combining the results of Theorems \ref{linsys1}, \ref{as} and \ref{suffcond} we can construct  \emph{explicit}  signature $(4,3)$ metrics with \emph{full}  ${\bf G}_2$ holonomy which are given in  closed form, as polynomials in the variables  $p$ and $\varrho$, provided that the function $f=f(x,y,p)$ is a polynomial in $p$ of order $k\geq 4$. 
Here is an example of such a full ${\bf G}_2$ holonomy metrics that are polynomial in $p$ and $\varrho$:
\begin{example}\label{polyEx}
On $M=\{(x,y,z,p,q)\}$ let $f_0,f_1,\dots f_9$ be arbitrary real analytic functions of the variables $x$ and $y$. Define a function $f$ of $x,y$ and $p$ as
$$
f=f_0+f_1p+f_2p^2+f_3p^3+f_4p^4+f_5p^5+f_6p^6+f_7p^7+f_8p^8+f_9p^9.
$$
Then Theorem \ref{as} implies that the ambient metric for $[g_{{\mathcal D}_{f,0}}]$ 
is
\be
{\color{black}\tilde{g}=2\der t\der(\varrho t)+t^2\Big(}2\omega^1\omega^5-2\omega^2\omega^4+(\omega^3)^2{\color{black}+A\, (\omega^1)^2+2B\, \omega^1\omega^4+C\, (\omega^4)^2\Big)},\label{am1}\ee
where the $\omega^i$'s are defined in \eqref{omegas} and the functions $A$, $B$ and $C$ given on $\tilde{M}$ by
$$
\begin{aligned}
&A=\tfrac{63}{8}(f_8+9pf_9)\varrho^3+\tfrac{27}{8}(f_6+7pf_7+28p^2f_8+84p^3f_9)\varrho^2-\\&\quad\quad\quad\tfrac{9}{5}(f_4+5pf_5+15p^2f_6+35p^3f_7+70p^4f_8+126p^5f_9)\varrho,\\
&B=-\tfrac{63}{256}f_9\varrho^4-\tfrac{7}{64}(f_7+8pf_8+36p^2f_9)\varrho^3+\tfrac{1}{16}(f_5+6pf_6+21p^2f_7+56p^3f_8+126p^4f_9)\varrho^2-\\
&\quad\quad\quad\tfrac{3}{20}(f_3+4pf_4+10p^2f_5+20p^3f_6+35p^4f_7+56p^5f_8+84p^6f_9)\varrho,\\
&C=\tfrac{7}{1152}(f_8+9pf_9)\varrho^4+\tfrac{1}{360}(f_4+5pf_5+15p^2f_6+35p^3f_7+70p^4f_8+126p^5f_9)\varrho^2+\\
&\quad\quad\quad\tfrac{1}{45}(2b^2-f_2-3pf_3-6p^2f_4-10p^3f_5-15p^4f_6-21p^5f_7-28p^6f_8-36p^7f_9)\varrho.
\end{aligned}
$$
For a generic choice of the  functions $f_0,\dots, f_9$,  the  metric \eqref{am1} has holonomy equal to ${\bf G}_2$.
\end{example}

Next we present two more complicated examples. These will be examples of ambient metrics $\tilde{g}_{{\mathcal D}_{f,h}}$ with full ${\bf G}_2$ holonomy, which are also in closed form, but which are {\em not  polynomial} in the variable $p$. We construct them by finding special closed form solutions to the second order system \eqref{2sys}. The first of the examples gives a strategy how to find solutions of this type.
\begin{example}
  Consider the distribution $\mathcal D_{f, h}$ with $f(x, y, p) = \sin(p) $ and $h(x,y) =0$.  

In order to find the ambient metric for $g_{\mathcal D_{f, h}}$, we first solve the
equation $L(A) = f_{pppp}$ for $A(x,y, p ,\varrho)$ and then determine $B$ and $C$ from the first order equations \eqref{1sys} and the initial conditions $B|_{\varrho=0}=0$ and $C|_{\varrho=0}=0$.  To solve for $A$, we remark that
if $f=f(x,y,p)$ satisfies a constant coefficient linear ODE  with respect to the variable $p$ and with fundamental solutions $f_1(x,y,p),  \ldots , f_N(x,y,p)$,
then the analytic solutions of  \[L(A)  = f_{pppp}\]
are of the form
\[
A = a_1(\varrho)f_1(x,y,p) + a_2(\varrho)f_2(x,y,p)  +  \ldots  + a_N(\varrho)f_N(x,y,p)
\]
where the  $a_2(\varrho), \ldots , a_N(\varrho)$ satisfy a system of linear second order ODEs.

For the problem at hand, this implies that  $A = a_1(\varrho)\sin(p) + a_2(\varrho)\cos(p)$, where
\begin{equation*}
   2\varrho a_{1,\varrho\varrho}-3 a_{1,\varrho} + \tfrac{1}{8} a_1 = \tfrac{9}{40},\quad
   \varrho a_{2,\varrho\varrho}-3 a_{2,\varrho}+\tfrac{1}{8} a_2  = 0 \quad\text{and}\quad a_1(0) = a_2(0)= 0,
\end{equation*}
The analytic solutions to these  equations are 
\begin{equation*}
   a_1(\varrho)= \tfrac{3}{20}\varrho \cos(\tfrac{\sqrt{\varrho}}{2})- \tfrac{9}{10}\sqrt{\varrho}\sin(\tfrac{\sqrt{\varrho}}{2})-\tfrac{9}{5}( \cos(\tfrac{\sqrt{\varrho}}{2})- 1\quad\text{and}\quad a_2(\varrho) = 0.
\end{equation*} 
Then we find that  
$$B =b(\varrho)\cos(p)\quad {\rm and}\quad C=c(\varrho)\sin(p),$$ 
where  
\begin{eqnarray*}
b(\varrho) 
&= &
-
\tfrac{1}{120}\varrho^{\frac{3}{2}} \sin(\tfrac{\sqrt{\varrho}}{2})
- \tfrac{1}{10}\varrho\cos(\tfrac{\sqrt{\varrho}}{2})
+\tfrac{1}{2}\sqrt{\varrho}\sin(\tfrac{\sqrt{\varrho}}{2})
+\cos(\tfrac{\sqrt{\varrho}}{2})
-1
\\
c(\varrho) 
&= &
\tfrac{1}{2160}\varrho^{2} \cos(\tfrac{\sqrt{\varrho}}{2})
-
\tfrac{1}{120}\varrho^{\frac{3}{2}} \sin(\tfrac{\sqrt{\varrho}}{2})
- \tfrac{13}{180}\varrho\cos(\tfrac{\sqrt{\varrho}}{2})
+\tfrac{1}{3}\sqrt{\varrho}\sin(\tfrac{\sqrt{\varrho}}{2})
+\tfrac{2}{3}(\cos(\tfrac{\sqrt{\varrho}}{2})
-1).
\end{eqnarray*}
Note that  these functions are analytic in $\varrho$. Moreover, by Theorem  \ref{suffcond} the holonomy of the  resulting ambient metric given by $A$, $B$ and $C$ is equal to $\mathbf{G}_2$.
\end{example}

\begin{example}\label{expEx}
Similarly, for
\[F=q^2+\e^p.\]
the solutions of  equations \eqref{2sys} are of the form
\[
A=A(\varrho, p) =a(\varrho)\e^p,\ \ \ B=B(\varrho, p) =b(\varrho)\\e^p,\ \ \ C=C(\varrho, p) =c(\varrho)\e^p
\]
for functions $a,b,c$ of $\varrho$ given by
\begin{eqnarray*}
a(\varrho) 
&= &
\tfrac{3}{20}\varrho \cosh(\tfrac{\sqrt{\varrho}}{2})- \tfrac{9}{10}\sqrt{\varrho}\sinh(\tfrac{\sqrt{\varrho}}{2})-\tfrac{9}{5}( \cosh(\tfrac{\sqrt{\varrho}}{2})-1)
\\
b(\varrho) 
&= &
-
\tfrac{1}{120}\varrho^{\frac{3}{2}} \sinh(\tfrac{\sqrt{\varrho}}{2})
+ \tfrac{1}{10}\varrho\cosh(\tfrac{\sqrt{\varrho}}{2})
-\tfrac{1}{2}\sqrt{\varrho}\sinh(\tfrac{\sqrt{\varrho}}{2})
+\cosh(\tfrac{\sqrt{\varrho}}{2})
-1
\\
c(\varrho) 
&= &
\tfrac{1}{2160}\varrho^{2} \cosh(\tfrac{\sqrt{\varrho}}{2})
-
\tfrac{1}{120}\varrho^{\frac{3}{2}} \sinh(\tfrac{\sqrt{\varrho}}{2})
+ \tfrac{13}{180}\varrho\cosh(\tfrac{\sqrt{\varrho}}{2})
-\tfrac{1}{3}\sqrt{\varrho}\sinh(\tfrac{\sqrt{\varrho}}{2})
+\tfrac{2}{3}(\cosh(\tfrac{\sqrt{\varrho}}{2})
-1).
\end{eqnarray*}
Again, by Theorem  \ref{suffcond} we conclude that the holonomy of the   ambient metric is equal to $\mathbf{G}_2$.
\end{example}

Finally we give an example that shows that even when the sufficient condition $A_\varrho\not=0$ in Theorem \ref{suffcond} is not satisfied, we can obtain ambient metrics with holonomy equal to $\mathbf{G}_2$.
\begin{example} Consider the conformal class given by $f\equiv 0$. For such a conformal class the ambient metric in \eqref{amf} has $A\equiv 0$, 
$B=
-\frac{3}{20}\varrho h_y$, and $
C=
\frac{2}{45}\varrho h^2
-\frac{1}{15}\varrho (ph_y+h_x) 
$. When taking $h$ as simple as $h(x,y)=y$ and differentiating the curvature, we get that the ambient holonomy has dimension $14$ and hence is equal to $\mathbf{G}_2$.
\end{example}

\subsection{Solutions \trev{non-smooth} in $\varrho$}\label{nona-sec}
To find \emph{all} solutions to the linear system (\ref{2sys}), i.e., also the ones that are \trev{not smooth}  in $\varrho$,  
in analogy to Theorem \ref{pptheo} we observe that the two independent solutions to $L(\varrho^k)=0$ are $\varrho^0$ and $\varrho^{ 5/2}$.  Thus, the most general solution to the  system \eqref{2sys} can be obtained by the following series
 \begin{eqnarray*}
 A&=&\sum_{k=1}^\infty a_k(x,y,p)\varrho^k+\varrho^{ 5/2}\sum_{k=0}^\infty \alpha_k(x,y,p)\varrho^k,\\
B&=&\sum_{k=1}^\infty b_k(x,y,p)\varrho^k+\varrho^{ 5/2}\sum_{k=0}^\infty \beta_k(x,y,p)\varrho^k,\\
C&=&\sum_{k=1}^\infty c_k(x,y,p)\varrho^k+\varrho^{ 5/2}\sum_{k=0}^\infty \gamma_k(x,y,p)\varrho^k.
 \end{eqnarray*}
\begin{theorem}\label{nasol}
The general solution to the linear system \eqref{2sys}
is given by
\beqn
A
&=&
\frac{3}{5} \sum_{k=1}^\infty\frac{(2k-1)(2k-3)}{2^{2k}(2k)!}\,\frac{\partial^{(2k+2)}{ f}}{\partial p^{(2k+2)}}\, \varrho^k+
60\, \varrho^{ 5/2} \sum_{k=0}^\infty \frac{(k+2)(k+1)}{2^{2k}(2k+5)!}\,\frac{\partial^{2k}{ \alpha_0}}{\partial p^{2k}}\,\varrho^k,\\
B&=&
-\frac{3}{20}\varrho h_y
-\frac{1}{15}\sum_{k=1}^\infty\frac{(2k-1)(2k-3)(2k-5)}{2^{2k}(2k)!}\,\frac{\partial^{(2k+1)}{ f}}{\partial p^{(2k+1)}}\, \varrho^k
\\&&
{}+\frac{20}{3}\varrho^{ 5/2}\sum_{k=0}^\infty {(k+1)(k+2)}{2^{2k}(2k+5)!}\,\Big(9\frac{\partial^{2k}{ \beta_0}}{\partial p^{2k}}-2k\frac{\partial^{(2k-1)}{ \alpha_0}}{\partial p^{(2k-1)}}\Big)\,\varrho^k,\\
C&=&
\frac{2}{45}\varrho h^2
-\frac{1}{15}\varrho (ph_y+h_x)
+
\frac{2}{135} \sum_{k=1}^\infty\Big(\frac{(k-3)(2k-1)(2k-3)(2k-5)}{2^{2k}(2k)!}\,\frac{\partial^{2k}{ f}}{\partial p^{2k}}\Big)\, \varrho^k
\\&&{}+ \frac{20}{27}\,\varrho^{ 5/2}\sum_{k=0}^\infty\frac{(k+1)(k+2)}{2^{2k}(2k+5)!}\,\Big(81\frac{\partial^{2k}{ \gamma_0}}{\partial p^{2k}}-36k\frac{\partial^{(2k-1)}{ \beta_0}}{\partial p^{(2k-1)}}+2k(2k-1)\frac{\partial^{(2k-2)}{ \alpha_0}}{\partial p^{(2k-2)}}\Big)\,\varrho^k.
\eeqn
Here  ${ \alpha_0},{ \beta_0}$ and ${ \gamma_0}$, are \emph{arbitrary} functions of the variables $x$, $y$ and $p$.
\end{theorem}

Note that the analytic solutions are totally determined by the distribution, i.e. by the functions $f$ and  $h$. On the other hand, the  non-smooth part of the solutions is determined only by the functions $\alpha_0$, $\beta_0$ and $\gamma_0$ and hence does  not depend on the distribution ${\mathcal D}_{f,h}$ at all.

\begin{remark}\label{1sysrem}
In relation to Remark \ref{na-rem} we observe that the non-smooth solutions to the second order system \eqref{2sys} that are given in Theorem \ref{nasol}, when taking $f=h=0$ and $\alpha_0=\alpha_0(x,y)$,  do {\em not} satisfy the first order system \eqref{1sys}. Indeed, in this case we get $A\equiv 0$ and the first two equations of 
\eqref{1sys} become $B_p=0$ and $B_\varrho=0$, which do not hold when choosing
 $\frac{\partial^2\beta_0}{\partial p^2}\not=0$.
\end{remark}

\subsection{Remarks on the holonomy of the \trev{non-smooth} solutions}
The general solution to the linear system (\ref{2sys}) given in Theorem \ref{nasol} 
 enables us, via \eqref{amf}, to write down  {explicit} Ricci-flat metrics $\tilde{g}_{{\mathcal D}_{f,h}}$ that truncate at a prescribed order in the variable $\varrho$.
Here we use some of these solutions to indicate certain holonomy issues related to the ambient metric construction. 
Note that  solutions with nontrivial $\varrho^{\frac{5}{2}+k}$ terms are only defined for $\varrho\geq 0$ and that they are only twice differentiable at $\varrho=0$. Thus, considering non-smooth in $\varrho$ ambient metrics and trying to study the holonomy questions we would need a notion of a holonomy of a pseudo-Riemannian manifold with a boundary. We are unaware of such setting for the holonomy. Therefore in the following, considering our examples, we will only talk about the holonomy of $\tilde{g}_{{\mathcal D}_{f,h}}$ on a domain where $\varrho>0$.

In the  discussion below we concentrate on the non-smooth ambient metrics (\ref{amf}) for the \emph{flat} distribution for which $f\equiv 0$ and $h\equiv 0$.  
In such case the solution to the linear system (\ref{2sys}) is
$$\begin{aligned}
A&=\varrho^{ 5/2} \sum_{k=0}^\infty 60\frac{(k+2)(k+1)}{2^{2k}(2k+5)!}\frac{\partial^{2k}{ \alpha_0}}{\partial p^{2k}}\, \varrho^k,\\
B&=\varrho^{ 5/2}\sum_{k=0}^\infty \frac{20}{3}\frac{(k+1)(k+2)}{2^{2k}(2k+5)!}\Big(9\frac{\partial^{2k}{ \beta_0}}{\partial p^{2k}}-2k\frac{\partial^{(2k-1)}{ \alpha_0}}{\partial p^{(2k-1)}}\Big)\, \varrho^k,\\
C&=\varrho^{ 5/2}\sum_{k=0}^\infty \frac{20}{27}\frac{(k+1)(k+2)}{2^{2k}(2k+5)!}\Big(81\frac{\partial^{2k}{ \gamma_0}}{\partial p^{2k}}-36k\frac{\partial^{(2k-1)}{ \beta_0}}{\partial p^{(2k-1)}}+2k(2k-1)\frac{\partial^{(2k-2)}{ \alpha_0}}{\partial p^{(2k-2)}}\Big)\,\varrho^k,
\end{aligned}$$
 and depends on three arbitrary functions $\alpha_0,\beta_0,\gamma_0$ of variables $x$, $y$ and $p$.
 As an illustration we discuss holonomy properties of the corresponding ambient metrics for a very simple example, in which we make a particular choice of these 3 functions.

\begin{example}
Let  $c$ be a real constant and set
$$\begin{aligned}
\alpha_0=&\beta(x)+p\alpha(x),\quad\quad\beta_0=\varphi_0(x)+p\varphi_1(x)+252c p^2\alpha(x),\\
\gamma_0=&\varphi_3(x)+p\varphi_4(x)+\tfrac{1}{81}p^2(2268\varphi_2(x)-\beta(x)+18\varphi_1(x)).
\end{aligned}$$
This gives the following solution to the linear system (\ref{2sys}),
$$
\begin{aligned}
A&=\Big(252p\alpha(x)+\beta(x)\Big)\varrho^{5/2},\\
B&=(9c-1)\alpha(x)\varrho^{7/2}+\Big(\varphi_0(x)+p\varphi_1(x)+252 c p^2\alpha(x)\Big)\varrho^{5/2},\\
C&=\Big(\varphi_2(x)+\big(\tfrac19-4c\big)p\alpha(x)\Big)\varrho^{7/2}
+\Big(\varphi_3(x)+p\varphi_4(x)+\tfrac{1}{81}p^2\big(18\varphi_1(x)+2268\varphi_2(x)-\beta(x)\big)\Big)\varrho^{5/2},
\end{aligned}
$$
corresponding to the following  \emph{family} of ambient metrics for the \emph{flat} conformal structure 
\be
\begin{aligned}
\tilde{g}=\ &2\der t\der(\varrho t)+t^2\Big(}8(\der p-q\der x)^2-6(\der z-2q\der p+q^2\der x)\der x-12(\der y-p\der x)\der q{\color{black}+\\&+(252p\alpha+\beta)\varrho^{5/2}\, (\der y -p\der x)^2+\Big.
\\&+6((9c-1)\alpha\varrho^{7/2}+(\varphi_0+p\varphi_1+252 c p^2\alpha)\varrho^{5/2})\,  (\der y -p\der x)\der x+\\&
\Big. +9((\varphi_2+(\tfrac19-4c)p\alpha)\varrho^{7/2}+(\varphi_3+p\varphi_4+\tfrac{1}{81}p^2(18\varphi_1+2268\varphi_2-\beta))\varrho^{5/2})\, \der x^2\Big).\end{aligned}.\label{flatn}\ee
We see that the obtained family of Ricci-flat metrics $\tilde{g}$ depends on \emph{seven} arbitrary functions $\alpha=\alpha(x), \beta=\beta(x), \varphi_0=\varphi_0(x),\varphi_1=\varphi_1(x),\dots,\varphi_4=\varphi_4(x)$ of the variable $x$, and a real constant $c$. The metric $\tilde{g}$ is an ambient metric for the flat conformal structure represented by a flat metric 
$$g=8(\der p-q\der x)^2-6(\der z-2q\der p+q^2\der x)\der x-12(\der y-p\der x)\der q.$$ 
The holonomy properties of the family $\tilde{g}$ are quite interesting. Generically
 these metrics have \emph{full} $\mathfrak{so}(4,3)$ holonomy. We illustrate the behaviour with a few special cases:
\begin{enumerate}
\setlength{\itemsep=0mm}
\setlength{\labelsep=2pt}
\item In order to verify the statement about  generic holonomy, 
  we put
 \[\beta(x)\equiv \varphi_0(x)\equiv \varphi_1(x)\equiv \dots\equiv \varphi_4(x)\equiv 0,\ \alpha(x)\equiv 1,\]
 and $c$ an arbitrary constant. Then  our calculations   show the following:
  The   dimensions of the
  vector spaces spanned by the derivatives of the curvature at the point with $x=y=z=p=0$ and $q=\varrho=t=1$  up to order  4 are  6,  15, 18, 20, 21, respectively.
Hence, even in this special situation, the holonomy algebra   must be equal to $\mathfrak{so}(4,3)$. 
\item
Interestingly, the metrics $\tilde{g}$ include, as special cases, metrics with holonomy {\em equal to}  ${\bf G}_2$.
For this we put
\[\beta(x)\equiv \varphi_0(x)\equiv \varphi_1(x)\equiv \dots\equiv \varphi_4(x)\equiv 0,\ \ c=0.\] 
Here our  calculations show   the following:
The   dimensions of the vector spaces spanned by the derivatives of the curvature at the point with $x=y=z=p=0$ and $q=\varrho=t=1$  up to order 5 are 4,  10, 11, 12, 13, 14, respectively, and do not increase when differentiating again. 
 Hence, the holonomy algebra is $14$-dimensional, and one can check that it is equal to ${\mathfrak g}_2$. 
\item
Another choice of the free functions in (\ref{flatn}) shows that the considered family of Ricci flat metrics may still have different holonomy algebra. For this
 we set \[\alpha(x)\equiv\beta(x)\equiv \varphi_0(x)\equiv \varphi_1(x)\equiv \varphi_2(x)\equiv \varphi_4(x)\equiv 0,\ \ c=1/9.\] 
 Here our Maple calculations show that the 
 dimensions of the
vector spaces spanned by the derivatives of the curvature up to order 2 are 4, 9, 10, respectively, and do not increase when differentiating again. 
In fact, the holonomy algebra   is a semidirect product of a $7$-dimensional radical and $3$-dimensional semisimple Lie algebra. It is \emph{not} a subalgebra of ${\mathfrak g}_2$. 
 \end{enumerate}
\end{example}

\section{Generic 3-distributions in dimension 6 and their ambient metrics}
\label{secspin43}
\subsection{Generic 3-distributions in dimension 6 and the associated conformal classes}
In this section we construct explicit examples of ambient metrics for  conformal structures  with metrics of signature $(3,3)$ which are naturally associated to \emph{generic} rank 3 distributions in dimension~$6$. 
Here \emph{generic} means that the distribution satisfies 
$$[{\mathcal D},{\mathcal D}]+{\mathcal D}={\rm T}M^6.$$
The associated conformal structures were introduced by Bryant in \cite{bryant06} as structures that encode local invariants of such distributions. We will call them {\em Bryant's conformal structures}.

Since our purpose is to find new examples of ambient metrics, we will not consider all generic rank $3$ distributions in dimension $6$, but only a special subclass. 
We consider an open set $\mathcal U$ near the origin in $\bbR^6$ with coordinates $(x^1,x^2,x^3,y^1,y^2,y^3)$. We  define a rank $ 3$ distribution ${\mathcal D}_f$ as the annihilator of three 1-forms
$$\theta_1=\der y^1+x^2\der x^3,\quad\theta_2=\der y^2+f\der x^1,\quad\theta_3=\der y^3+x^1\der x^2.$$
Here  $f=f(x^1,x^2,x^3)$ is a differentiable function of the variables $(x^1,x^2,x^3)$ only. 
Explicitly, the distribution is defined by
$${\mathcal D}_f=\Span\left(\,\frac{\partial}{\partial x^3}-x^2\frac{\partial}{\partial y^1},\,\,\,\frac{\partial}{\partial x^1}-f\frac{\partial}{\partial y^2},\,\,\,\frac{\partial}{\partial x^2}-x^1\frac{\partial}{\partial y^3}\,\right).$$
This distribution is {generic} provided that $$\frac{\partial f}{\partial x^3}\neq 0,$$ 
and this will be always assumed in what follows.
In the special case of a 3-distribution ${\mathcal D}_f$ defined above, Bryant's conformal class $[g_{{\mathcal D}_f}]$ can be represented by a metric $g_{{\mathcal D}_f}$ as
\be
g_{{\mathcal D}_f}=2\theta^1\theta^4+2\theta^2\theta^5+2\theta^3\theta^6,\label{bm}\ee
with the null co-frame
\be
\begin{aligned}
\theta^1=&~\der y^1+x^2\der x^3,\qquad \theta^2=\der y^2+f\der x^1,\qquad\theta^3=~\der y^3+x^1\der x^2,\qquad
\theta^4=~\tfrac32 f_{3}^2~\der x^1,\\
\theta^5=&~\tfrac32 f_{3}~\der x^2~-~\tfrac12 f_{33}~\theta^1,\\
\theta^6=&~\tfrac32 f_{3}^2~\der x^3~+~\tfrac12f_{13}~\theta^2~+~\tfrac32 f_{3}f_{2}~\der x^2 ~+~\tfrac12(f_{3}f_{23}-f_{2}f_{33})~\theta^1~+~
\tfrac12(f_{2}f_{13}-f_{3}f_{12})~\theta^3.
\end{aligned}\label{bm1}\ee
In these formulae  we denote the partial derivatives with respect to the variables $x^i$ by a subscript, for example, $f_1=\tfrac{\partial f}{\partial x^1}$ or  $f_{13}=\tfrac{\partial^2 f}{\partial x^3\partial x^1}$. 
We should point out that, if we scale the first three forms in this co-frame by $\frac{1}{f_{3}}$ and the last three by  $-\frac{3}{2} f_{3}$, then we obtain a co-frame which satisfies the structure equations
(2.17) in  \cite{bryant06}.

We will now restrict to  the distributions $\Cal D_f$ with $f=f(x^1,x^3)$, however all the following results will also be true for  $f=f(x^2,x^3)$. The metrics in \eqref{bm} defined by $f=f(x^1,x^3)$ have  properties that are remarkably similar to those of pp-waves. To describe these, we
denote by $\cA$ the algebra of functions depending only on $x^1$ and $x^3$. 
\begin{lemma}\label{newprop}
The metric  $g_{{\mathcal D}_f}$ in \eqref{bm}  defined by $f\in \cA$ has the following  properties:
\begin{enumerate}
\item For $k\in \{2,4,6\}$ we have
$ \nabla \theta^k\in \spA\{\theta^i\otimes \theta^j\mid i,j\in \{2,4,6\}\}  $.
\item 
The $3$-form 
\[\kappa=(f_3)^{-1/3}\theta^2\wedge \theta^4\wedge\theta^6\]
is parallel, $\nabla \kappa=0$.
\item The rank $3$ distribution $\Cal V:=\{X\in T\Cal U\mid X\hook \kappa=0\}$ is parallel, i.e.,
\[\nabla_XV \in \Gamma (\Cal V),\quad\text{ for all $V\in  \Gamma (\Cal V)$ and all $X\in T\Cal U$}.\]
In particular, 
$\Cal V$ is integrable and, in fact, spanned by $\frac{\partial}{\partial x^2}$, $\frac{\partial}{\partial y^1}$ and $\frac{\partial}{\partial y^3}$.
\end{enumerate}
\end{lemma}
\begin{proof}
To check the  first point is a direct computation. For (ii), take a function $h=h(x^1,x^3)$, set $\phi=h\,\theta^2\wedge \theta^4\wedge \theta^6$, and compute
\[
\nabla \phi
= \frac{2}{9f_3^3} \left( 
(3h_3f_3+h f_{33})( 2\theta^6-   f_{13}\theta^2)
+
2(3h_1f_3+h f_{13})\theta^4\right)\otimes \theta^2\wedge\theta^4\wedge \theta^6.
\]
Then $h=(f_3)^{-1/3}$ gives the required $\kappa$. Since $\kappa$ is parallel, the distribution $\Cal V$ is parallel as well.
\end{proof}

\subsection{Bryant's conformal structures  with vanishing obstruction and their ambient metrics}
Recall that in even dimensions $n$, not every conformal class $[g]$ admits a smooth ambient metric. When proving their results in \cite{fefferman/graham85,fefferman-graham07}, Fefferman and Graham fixed 
 a metric $g_0$ in the conformal class and performed    the 
\emph{power series expansion} of $g(x^i,\varrho)$ in the variable $\varrho$, 
\be\label{expansion}
g(x^i,\varrho)=g_0+2\varrho P +\varrho^2\mu+\ldots.\ee
It follows that 
$P$ is the Schouten tensor
 for $g_0$ and, if $n>4$, that $\mu$ is given by
\be\label{mu}\mu_{ij}=\frac{1}{4-n}B_{ij}+P_i^{~k}P_{kj}.\ee
Here $B$ is the \emph{Bach tensor} of the metric $g_0$. It is defined by
$$B_{ij}=\nabla^k C_{ijk}-P^{kl}W_{kijl},$$
where $W_{ijkl}$ is the \emph{Weyl tensor} and 
$C_{ijk}$
 the {\em Cotton tensor} of $g_0$,
$$C_{ijk}=\nabla_kP_{ij}-\nabla_jP_{ik}.$$

In dimension $n=6$, for the power expansion of $g(x,\varrho)$ to continue beyond the $\varrho^2$ term, the metric $g_0$ representing the conformal class $[g_0]$ must satisfy 
\be\begin{aligned}\label{ot}
{\mathcal O}_{ij}&=\nabla^k\nabla_kB_{ij}-2W_{kijl}B^{kl}-4P_k^{~k}B_{ij}+8P^{kl}\nabla_lC_{(ij)k}-4C^{k~l}_{~i}C_{ljk}+\\&\quad\quad\quad +
2C_i^{~kl}C_{jkl}+4C_{ij}^{\,\,\,\, l}\nabla_lP^k_{~k}-4W_{kijl}P^k_{~m}P^{ml}\ \equiv\  0.\end{aligned}
\ee
In particular, when the {\em obstruction tensor} ${\mathcal O}_{ij}$ vanishes, the metric $g(x,\varrho)=g_0+2P\varrho+\mu\varrho^2$, when inserted in the formula \eqref{ambient-intro} for $\tilde{g}$, provides an example of an ambient metric $$
\tilde{g}=2\der(\varrho t)\der t+t^2(g_0+2P \varrho+\mu\varrho^2),
$$ that is, satisfying $Ric(\tilde{g})= 0$.

We will now return  to  the distributions $\Cal D_f$ with $f=f(x^1,x^3)$, i.e. $f\in \cA$, because for these 
the obstruction tensor for conformal classes of metrics  in \eqref{bm} vanishes. Again, all the following remains true for $f=f(x^2,x^3)$.
\begin{proposition}\label{p1}
In the case when $f=f(x^1,x^3)$  the conformal classes $[g_{{\mathcal D}_f}]$ 
defined by the metrics $g_{{\mathcal D}_f}$ have the following properties:
\begin{enumerate}
\setlength{\labelsep=2pt}
\item Generically they are not conformally Cotton, and hence not conformally Einstein.
\item They are Bach-flat and their Schouten tensor squares to zero, $P_{i}^{~k}P_{kj}=0$.
 In particular, the tensor $\mu $ in \eqref{mu} of the quadratic terms in  the expansion \eqref{expansion} vanishes.
\item Each term in the formula \eqref{ot} for the  Fefferman-Graham obstruction tensor $\mathcal O$  vanishes separately. Hence, $\Cal O\equiv 0$. 
\item An ambient metric for all such conformal classes is given by \[\tilde{g}_f=2\der(\varrho t)\der t+t^2(g_{{\mathcal D}_f}+2\,\varrho\,P).\]
\end{enumerate}
\end{proposition}
\begin{proof}
The fundamental observation is that 
the Schouten tensor $P$ of $g_{{\mathcal D}_f}$ with $f\in \cA$ satisfies\footnote{The explicit formulae for the Schouten tensor components are not needed here and can be found in the Maple worksheet at \url{http://digitalcommons.usu.edu/dg_applications/}.}
\begin{equation}
\label{roz1}
P\in \spA \{\theta^i\, \theta^j\mid i,j\in \{2,4,6\}\}
\end{equation}
and 
\[
\nabla P\in \spA \{\theta^i \otimes \theta^j\otimes \theta^k\mid i,j,k\in \{2,4,6\}\}.
\]
Property  \eqref{roz1} implies that the 
Schouten tensor is
 $2$-step nilpotent, $P_{ik}P^k_{~j}=0$.
Moreover, the Weyl   tensor $W$ satisfies 
\[ W(X,.,.,.)\in \mathrm{span} \{\theta^i \otimes \theta^j\otimes \theta^k\mid i,j,k\in \{2,4,6\}\}\]
whenever $X\in \Cal V$. These properties for the Schouten and the Weyl tensor imply the vanishing of the Bach tensor and of all other terms in the 
formula \eqref{ot} for $\Cal O$. 

That the metrics generically are not conformally Cotton-flat can be seen  by checking that generically there is no vector 
$\Upsilon^i$ such that
\[ C_{ijk}+\Upsilon^\ell W_{\ell ijk}=0,\]
which is a necessary condition for a metric being conformal to a Cotton-flat metric \cite{gover-nurowski04}.

The formula for an ambient metric then follows immediately from \eqref{expansion}. It is a
a straightforward computation to   check that $\tilde{g}_f$ is indeed Ricci-flat.
\end{proof}
Although  Proposition \ref{p1} gives us an explicit formula for an ambient metric for   Bryant's conformal classes
$[g_{{\mathcal D}_f}]$ with $f=f(x^1,x^3)$ or $f=f(x ^2,x^3)$,
 the ambient metrics are not unique.
In order to find additional ambient metrics and exhibit the ambiguity at order 3 in $\varrho$, 
we again make an ansatz for these higher order terms to have the same form as  Schouten tensor $P$ of $g_{D_f}$  as given in \eqref{roz1}.
Analogously to  pp-waves  and  generic rank 2 distributions in dimension $5$,  our ansatz is 
\begin{equation}\tilde{g}_{f,S}=2\der(\varrho t)\der t+t^2(g_{{\mathcal D}_f}+2\,\varrho\,P+S(x^1,x^3,\varrho)),\end{equation}
with the tensor $S=S(x^1,x^3,\varrho)$ as 
\be 
S(x^1,x^3,\varrho)=S_{22}(\theta^2)^2+2S_{24}\theta^2\theta^4+2S_{26}\theta^2\theta^6+S_{44}(\theta^4)^2+2S_{46}\theta^4\theta^6+S_{66}(\theta^6)^2\label{roz3}.\ee 
We now impose the Ricci-flatness condition on $\tilde{g}_{f,S}$ and remarkably get a  similar result as in Theorem~\ref{linsys1}, this time showing the ambiguity at the $\varrho^3$-term. A straightforward computation of the Ricci-tensor of $\tilde{g}_{f,S}$ yields:
\begin{theorem}\label{BryantAmbientTheo}
Let $$\tilde{g}_{f,S}=2\der(\varrho t)\der t+t^2(g_{{\mathcal D}_f}+2\,\varrho\,P+S(x^1,x^3,\varrho)),$$
be a metric
with $f=f(x^1,x^3)$, $g_{{\mathcal D}_f}$  a metric as in (\ref{bm})-(\ref{bm1}) with Schouten tensor $P$ as in (\ref{roz1}), 
 and with the tensor $S=S(x^1,x^3,\varrho)$ as in (\ref{roz3}). Then $ \tilde{g}_{f,S}$ 
is an ambient metric for the conformal class of Bryant's metrics $[g_{{\mathcal D}_f}]$ if and only if \emph{all} the functions $S_{ij}=S_{ij}(x^1,x^3,\varrho)$ satisfy  \emph{the same linear} PDE, $$\varrho\frac{\partial^2S_{ij}}{\partial \varrho^2}-2\frac{\partial S_{ij}}{\partial \varrho}=0.$$  The general solution $S$ satisfying the initial conditions $S_{ij}(x^1,x^3,0)= 0$,  
is given by $S=
\varrho^3 Q$ with 
$$
Q(x^1,x^3)=Q_{22}(\theta^2)^2+2Q_{24}\theta^2\theta^4+2Q_{26}\theta^2\theta^6+Q_{44}(\theta^4)^2+2Q_{46}\theta^4\theta^6+Q_{66}(\theta^6)^2,$$
 where the functions $Q_{ij}(x^1,x^3)$ are \emph{arbitrary}. Every   such $Q$ gives an ambient metric for $g_{\mathcal D_f}$ by
\begin{equation}\label{BryantAmbientmetric}
\tilde{g}_{f,Q}~=~2\der(\varrho t)\der t~+~t^2\Big(\, g_{{\mathcal D}_f}~+~2\,\varrho \,P~+~\varrho^3 \,Q(x^1,x^3)\, \Big),\end{equation}
with $P$ the Schouten tensor of $g_{{\mathcal D}_f}$ as in \eqref{roz1}.
\end{theorem}
\begin{remark}
There is an analogous Theorem for $[g_{{\mathcal D}_f}]$ with $f=f(x^2,x^3)$.
\end{remark}

The  tensor $Q$, which is trace free with respect to $g_{D_f}$,   is responsible for the ``Fefferman-Graham ambiguity'', i.e., for the ambiguity of  ambient metrics for conformal classes in even dimensions. In general, it appears in the curvature tensor and thus cannot be gauged away by coordinate transformations. 

\subsection{On the holonomy of the ambient metrics}

Now we will study the holonomy of ambient metrics for conformal classes $[g_{{\Cal D}_f}]$ with $f\in \cA$. For ambient metrics with $Q\not=0$ we get a similar result  as for pp-waves. For the ambient metrics with $Q=0$, we show that their holonomy is contained in $\mathbf{Spin}(4,3)$. Our terminology regarding invariance and stabilizers is explained in Remark \ref{terminology}.

\begin{proposition}\label{propnew2}
 Let $f=f(x^1,x^3)$ be a smooth function of two variables,   $[g_{D_f}]$ be   Bryant's conformal structure represented by a metric $g_{D_f}$ as in \eqref{bm}, and let
\[\tilde{g}_{f,Q}=2\der(\varrho t)\der t+t^2(g_{{\mathcal D}_f}+2\,\varrho\,P+ \varrho^3\, Q),\]
be one of its ambient metrics with $Q\in \spA \{ \theta^i\theta^j \mid i,j\in \{2,4,6\}\} $.
Then:
\begin{enumerate}
\item If $\kappa=(f_3)^{-1/3}\theta^2\wedge \theta^4\wedge\theta^6$ is the $3$-form defined in Lemma~\ref{newprop},
then the 
$4$-form 
\[\widetilde{\kappa}=-\frac{2\, t^3}{9}\,dt\wedge \kappa  \ =\ 
-\frac{2\, t^3}{9\,(f_3)^{1/3}}\, dt\wedge  \theta^2\wedge \theta^4\wedge\theta^6  ,
\]
is parallel, i.e, $\tnab\widetilde{\kappa}=0$.
\item 
The  totally null distribution $\widetilde{\Cal V}$  of tangent vectors with $X\hook \widetilde{\kappa}=0$, which is spanned by 
$\frac{\partial}{\partial x^2}$, $\frac{\partial}{\partial y^1}$, $\frac{\partial}{\partial y^3}$ and $\partial_\varrho$, 
 is parallel.
 \item The holonomy algebra of $\tilde{g}_{f,Q}$  is contained in the stabilizer  of $\widetilde{\kappa}$, 
\[\mathfrak{hol}(\tilde{g}_{f,Q})\ \subset\ 
\mathfrak{sl}_4\R \ltimes \Lambda_4\R =
 \left\{{\small
\begin{pmatrix}
X & Z 
\\[1mm]
0 &-X^\top
\end{pmatrix}}
\mid
 X\in \mathfrak{sl}_4\R, \ Z+Z^\top=0\right\}\subset \mathfrak{so}(4,4),
\]
where $\Lambda_4\R$ denotes the skew symmetric $4\times 4$ matrices\footnote{We avoid the notation $\soa(4)$ here, because it would obscure the fact that $\Lambda_4\R$ is an Abelian ideal in $\mathfrak{sl}_4\R \ltimes \Lambda_4\R $.}. The representation of $
\mathfrak{sl}_4\R$ on the Abelian ideal  $\Lambda_4\R$ is given by $X\cdot Z= XZ-(XZ)^\top$.
\end{enumerate}
\end{proposition}
\begin{proof}
The fact that $\kappa$ is parallel for $g_{{\Cal D}_f}$ suggests to consider the $4$-form $dt\wedge\kappa$. 
One checks that 
\[
\tnab(dt\wedge\kappa) =-\frac{3}{t}\, dt\otimes (dt\wedge\kappa).\]
Hence, rescaling $dt\wedge \kappa$ to $t^3 \,dt\wedge \kappa$ gives a parallel form (the factor $-2/9$ will become clear in the next proposition). 
Since $\widetilde{\kappa}$ is decomposable it defines a rank $4$ distribution $\widetilde{\Cal V}$ that is parallel. 
In fact, $\widetilde{\Cal V}=\R\cdot \del_\varrho\oplus \Cal V$, where $\Cal V$ was defined in Lemma \ref{newprop}. 
Note that $\widetilde{\Cal V}$ is totally null.
The parallel null distribution $\widetilde{\Cal V}$ is invariant under the holonomy and its stabilizer in $\soa(4,4)$ is 
$\mathfrak{gl}_4\R \ltimes \Lambda_4\R$. That $\widetilde \kappa$ is invariant under the    holonomy algebra,  reduces the holonomy further to $\mathfrak{sl}_4\R \ltimes \Lambda_4\R$.
\end{proof}
\begin{remark}\label{poremark}
Note that $\mathfrak{sl}_4\R\simeq \soa(3,3)$ as Lie algebra. Moreover,  the representation of
$\mathfrak{sl}_4\R$ on the Abelian ideal $\Lambda_4\R$ in the above
 $\mathfrak{sl}_4\R \ltimes \Lambda_4\R$ is just the representation of $\mathfrak{sl}_4\R$ on $\Lambda^2\R^{4}\simeq\Lambda_4\R$. The latter carries an $\mathfrak{sl}_4\R$-invariant bilinear form of signature $(3,3)$ defined by the relation \[\sigma \wedge \ \xi =\langle \sigma,\xi\rangle\, e_1\wedge e_2\wedge e_3\wedge e_4.\] Hence, we have
\[
\mathfrak{hol}(\tilde{g}_{f,Q})\ \subset\ 
\mathfrak{sl}_4\R \ltimes \Lambda_4\R \ \simeq \ \mathfrak{po}(3,3),\]
where $\mathfrak{po}(3,3)$ denotes the Poincar\'{e} algebra in signature $(3,3)$, 
$\mathfrak{po}(3,3)=
\soa(3,3)\ltimes \R^{3,3}$.

When writing $\mathfrak{hol}(\tilde{g}_{f,Q})\subset \mathfrak{po}(3,3)\subset \soa(4,4)$
we refer to this representation and {\em not} to the more familiar representation of $\mathfrak{po}(3,3)$ in $\soa(4,4)$ as the stabilizer of a null vector. The same applies in the  following, when we consider $\mathfrak{po}(3,2)=\soa(3,2)\ltimes \R^{3,2}\subset \mathfrak{po}(3,3)$.
\end{remark}
Since $\mathfrak{po}(3,3)$ and $\mathfrak{spin}(4,3)$ have the same dimension, but are different,  we obtain:
\begin{corollary}
Ambient  metrics \eqref{BryantAmbientmetric} for Bryant's conformal structures with $f=f(x^1,x^3)$ or $f=f(x^2,x^3)$  cannot have holonomy equal to $\mathfrak{spin}(4,3)$.
\end{corollary}
The next example shows that $\mathfrak{po}(3,3)$ is actually attained as holonomy group.
\begin{example}We set $f=x^1 (x^3)^2$ and  $Q=\theta^2 \theta^6$. In this case it turns out that the holonomy algebra of the ambient metric \eqref{BryantAmbientmetric}
is   $21$-dimensional. Hence, by Proposition 
\ref{propnew2}
it is the  Poincar\'{e} algebra in signature $(3,3)$,
\[\mathfrak{hol}_p(\tilde{g}_{f,Q})\simeq\mathfrak{po}(3,3)=\mathfrak{so}(3,3)\ltimes\R^{3,3}.
\]
\end{example}

\begin{remark}
Since the conformal holonomy is contained in the ambient holonomy, the result in \cite[Theorem 1]{Lischewski14} implies that a conformal classes $[g_{\Cal D_f}]$ with $f=f(x^1,x^3)$ contains a certain preferred  metric $g_0$. This metric  admits a parallel totally null rank $3$ distribution which contains the image of the Ricci-tensor (or equivalently, of the Schouten tensor). We have already established in Lemma \ref{newprop} and in the proof of Proposition \ref{p1} that the metric $g_{\Cal D_f}$ is equal  to this metric $g_0$, which shows that we have chosen a suitable conformal  factor.
\end{remark}

Now we turn to  ambient metrics with $Q=0$. From now on we will assume that
 $f=f(x^1,x^3)$  is a smooth function of two variables,   $[g_{D_f}]$    Bryant's conformal structure represented by a metric 
$g_{{\mathcal D}_f}=2\theta^1\theta^4+2\theta^2\theta^5+2\theta^3\theta^6$
as in \eqref{bm}, with Schouten tensor $P$, and that
\[\tilde{g}_{f}=2\der(\varrho t)\der t+t^2(g_{{\mathcal D}_f}+2\,\varrho\,P),\]
is the ambient metric with $Q=0 $.
Then, a direct computation shows:
\begin{lemma}\label{alphalemma}
The $2$-form
\[\alpha
=
-\frac{9\, t}{2} 
\, f_3^{4/3}\, dt\wedge \theta^2
+
\frac{ t^2}{ f_3^{5/3}}\,
\big( 
 \theta^4\wedge\theta^6
+
\frac{5}{2}\, f_{13}\,
\theta^2\wedge\theta^4
+
2
\, f_{33}\,  \theta^2\wedge\theta^6
 \big)\]
is parallel for $\tilde{g}_{f}$ and squares to $\widetilde{\kappa}$, i.e., $\alpha\wedge\alpha=\widetilde{\kappa}$. 
\elem
The key step in simplifying the calculations from now on is to introduce a co-frame $\omega^i$ in which  the metric $\tilde{g}_{f} $,  the parallel $2$-form $\alpha$, and  the parallel $4$-form $\widetilde{\kappa}$  can be written as follows:
\begin{eqnarray*}\tilde g_{f}& =&  \omega^1 \omega^5 + \omega^2 \omega^6 + \omega^3 \omega^7 + \omega^4 \omega^8,\\
 \alpha &= & \omega^5 \wedge \omega^8 + \omega^6 \wedge \omega^7,
\\
\widetilde{\kappa}&=&\omega^{5678},\qquad\text{where we denote $\omega^{ijhk} = \omega^i\wedge \omega^j\wedge \omega^h \wedge \omega^k$.}
\end{eqnarray*}
This means that the parallel null distribution $\widetilde{\Cal V}$ is spanned by $E_1, \ldots , E_4$, where $E_i$ is the dual frame to $\omega^i$, $\omega^i(E_j)=\delta^i_{~j}$.
Such a co-frame is given by
\begin{align*}
\omega^1 &=  \varrho\,dt + t\, d\varrho -\frac{t \varrho}{9\, f_{3}^3}\left( 5\, f_{13}\theta^4 + 4f_{33} \, \theta^6\right),
\\
\omega^2 &= 
t^2 \left(  
-\frac{5\,   f_{13}}{9\, f_3^3}\,  d\varrho + \theta^1 + \frac{\varrho \left(  85\, f_{33}^2-27  f_3\, f_{113} \right)}{  81\, f_3^6}\, \theta^4 \right),
\\
\omega^3 &= 
-\frac{4  \, f_{33}}{9\, f_3^{4/3}}\, d\varrho 
+ f_3^{5/3}\, \theta^3 
+ \frac{2\varrho \left(80f_{33}f_{13} -27\,f_3\, f_{133} \right)}{81\, f_3^{13/3}}\,  \theta^4
 + \frac{\varrho\left(76\,f_{33}^2 -27\,f_3\, f_{333}\right)}{81\, f_3^{13/3}}\,\theta^6
,
\\
\omega^4 &= 
t\, \Big(
s_1\,\theta^2\Big.
+s_2\,\theta^4 
-\frac{2}{9\,  f_3^{4/3}}\, \theta^5 
+ s_3\,\theta^6
\Big),
\\
\omega^5 &=  dt + \frac{t}{9\, f_3^3}\left(  5\, f_{13} \theta^4 + 4\,f_{33} \, \theta^6\right),  
\qquad
\omega^6 = \theta^4,
\qquad
\omega^7 =  \frac{t^2}{f_3^{5/3}} \, \theta^6,
\qquad
\omega^8 = -\frac{9\, f_3^{4/3}}{2}\, \theta^2,
\end{align*}
where the functions $s_1$, $s_2$ and $s_3$ are defined by
\begin{eqnarray*}
s_1&:=&\frac{\varrho\left( 5\,f_3\,  f_{13}^2 f_{333} - 20\, f_{13}^2\, f_{33}^2   +3\,f_3^2 \, f^2_{133}+4\, f_3\,f_{33}^2\, f_{113} -3\,f_3^2\, f_{113}\,f_{333}\right)}{486\, f_{3}^{22/3}},
\\
s_2&=&-
\frac{2\, \varrho \left(5\, f_3\, f_{13}\,f_{133} -20\, f_{33}\, f_{13}^2 
+4\, f_{3}\, f_{33}\, f_{113}\right)}
{243\, f_{3}^{22/3}},
\\
s_3&:=& 
-\frac{2\varrho\left( 5\, f_3\, f_{13}\, f_{333} -20\, f_{13}\, f_{33}^2 + 4\, f_{3}\, f_{33}\, f_{133}\right)}
{243\, f_{3}^{22/3}}.
\end{eqnarray*}
Using this co-frame, the computations simplify and we are able to determine the parallel $4$-forms and the holonomy of $\tilde{g}_{f}$ when $Q=0$. To this end, using the frame $\omega^i$, we define  the $4$ form 
\begin{equation}
\label{phihat}\hat\phi=\omega^{1256}+\, \omega^{1357}-\, \omega^{1458} - 2\, \omega^{1467} 
- 2\, \omega^{2358} -\, \omega^{2367}
 +\, \omega^{2468}+\, \omega^{3478},\end{equation}
and the function 
\begin{equation}
\label{fcts} s =\frac{2\, \varrho\left( 80\, f_{33}\, f_{13}-27\, f_3\, f_{133}\right) }{81\, f_3^{13/3}}. \end{equation}
Then we can show: 
\begin{proposition}
\label{propnew3}
Let $f=f(x^1,x^3)$  a smooth function of two variables,   $[g_{D_f}]$    Bryant's conformal structure represented by a metric 
$g_{{\mathcal D}_f}=2\theta^1\theta^4+2\theta^2\theta^5+2\theta^3\theta^6$
as in \eqref{bm}, and let
\[\tilde{g}_{f}=2\der(\varrho t)\der t+t^2(g_{{\mathcal D}_f}+2\,\varrho\,P),\]
be the ambient metric with $Q=0 $. Then: 
\begin{enumerate}
\item 
If $\widetilde{\kappa}=\omega^{5678}$ is the  parallel  $4$-form of Lemma \ref{alphalemma}, $\hat\phi$ the $4$-form in \eqref{phihat} and $s$ the function in \eqref{fcts}, then the $4$-forms 
\begin{align*}
\beta& = 
\alpha\wedge \left(\omega^{15}+\omega^{26}+\omega^{37}+\omega^{48}-s \, \omega^{58}\right)\ =\ 
\omega^{1567} - \omega^{2568} - \omega^{3578} + \omega^{4678}
-s \,  \widetilde{\kappa},\quad\text{and}\\
\phi &= \, \hat{\phi}+s\,\beta +\frac{s^2}{2}\, \widetilde{\kappa}
 \end{align*}
are   parallel, self-dual, and satisfy
\[\alpha\wedge \beta =0,
\qquad 
\alpha \wedge \phi =
-3 * \alpha=-3(\omega^{14}+\omega^{23})\wedge \widetilde{\kappa},\]
as well as
\[
\beta\wedge \beta= \beta\wedge \phi=0 ,\qquad \phi\wedge \phi= 14\ \omega^{12345678}.\]
\item The  holonomy algebra of $\tilde{g}_{f}$ is contained in 
$\mathfrak{po}(3,2)$,
where 
 $\mathfrak{po}(3,2)=\soa(3,2)\ltimes \R^{3,2}$ is the Poincar\'{e} algebra in signature $(3,2)$ represented as explained in Remark \ref{poremark}.
 \item The   holonomy algebra of $\tilde{g}_{f}$ is contained in $\mathfrak{spin}(4,3)$.
\end{enumerate}
\end{proposition}
%
%
%
%

\begin{proof}
To check that  $\beta$ and $\phi$ are  parallel is a direct computation. Note that, when $Q\not=0$, in general they are  not parallel.
  
Since the $2$-form $\alpha$ of  Lemma \ref{alphalemma} is parallel,  the holonomy algebra is contained
in its stabilizer, i.e. in 
$\mathfrak{sp}_2(\R)\ltimes  \Lambda_4\R$. Here  $\mathfrak{sp}_2(\R)$ is the symplectic Lie algebra in $\mathfrak{sl}_4(\R)$ and again
 $ \Lambda_4\R$ are the skew-symmetric $4\times 4$ matrices.
 From  $\alpha$ we obtain the skew adjoint endomorphism field $\Cal J$  by
 \[ \tilde{g}_{f}(\Cal J X,Y)=\alpha(X,Y).\]
Since $\alpha $ is parallel, $\Cal J$ is parallel as well. In the fixed co-frame $\omega^i$ and its dual $E_i$,  
$\Cal J$ is of the form 
\[\Cal J=E_4\otimes \omega^5+E_3\otimes \omega^6 -E_2\otimes \omega^7-E_1\otimes \omega^8,\]
or, when written as a matrix with respect to $E_i$, 
\[\Cal J=\begin{pmatrix} 0&J\\0&0\end{pmatrix},\quad\text{with }J={\scriptsize \begin{pmatrix}0&0&0&-1\\ 0&0&-1&0\\ 0&1&0&0\\1&0&0&0\end{pmatrix}}.\]
First we notice that $\mathfrak{sp}_2\R$ acts trivially on $J$ by its very definition. Secondly, since $\alpha\wedge \alpha\not=0 $, $J$ is non-degenerate with respect to the signature $(3,3)$ scalar product defined by the identification $\Lambda_4\R\simeq \Lambda^2\R^4$ (see Remark \ref{poremark}). Hence, $\Lambda_4\R$ splits invariantly under $\mathfrak{sp}_2\R$ as \[\Lambda_4\R= (\R\cdot J)^\perp\oplus \R\cdot J =\R^{3,2}\oplus \R.\] Recalling  the isomorphism $\mathfrak{sp}_2\R\simeq \soa(3,2)$, that is given by this representation, 
we obtain that the holonomy is contained in $(\mathfrak{so}(3,2)\ltimes \R^{3,2})\oplus \R$.
Finally, with $\Cal J\cdot \alpha=0$, one can check that 
\[\Cal J\cdot \beta= \alpha\wedge \Cal J\cdot \left( \omega^{15}+\omega^{26}+\omega^{37}+\omega^{48}-s \, \omega^{58} \right)
=-2\, \alpha\wedge \alpha =-2\, \widetilde{\kappa}
\not=0.\] 
Since $\beta$ is parallel and hence holonomy invariant, this implies that $\Cal J$ is not contained in the holonomy algebra, which is then further reduced to  $\mathfrak{po}(3,2)=\mathfrak{so}(3,2)\ltimes \R^{3,2}$.

In order to check that the holonomy algebra is contained in $\mathfrak{spin}(4,3)$ we introduce the
  co-frame 
\begin{align*}
\xi^1 &= -\frac {a}{64}\,\omega^1+a\,\omega^5,\quad
&\xi^2 &= \frac{a}{16}\,\omega^2 - \frac{a}{4}\,\omega^6,\quad
&\xi^3 &= \frac{a}{8}\,\omega^3- \frac{a}{8}\,\omega^7,
\\
\xi^4 &=\frac{a}2\,\omega^4- \frac{a}{32}\,\omega^8,\quad
&\xi^5 &= \frac{a}{16}\, \omega^2 + \frac{a}{4}\,\omega^6, \quad
&\xi^6 &=-\frac{a}{64}\,\omega^1 -a\,\omega^5,
\\
\xi^7 &= -\frac{a}{2}\,  \omega^4 - \frac{a}{32}\,\omega^8, \quad 
&\xi^8 &= - \frac{a}{8}\, \omega^3 - \frac{a}{8}\, \omega^7,
&&\quad \text{with $a = \sqrt{32}$.}
\end{align*}
This transforms the metric $\tilde g_{f}$ and the $4$-form $\hat\phi$ into
\begin{align*}
\tilde g_{f} =\  &-(\xi^1)^2 - (\xi^2)^2 - (\xi^3)^2 - (\xi^4 )^2 + 
(\xi^5)^2 + (\xi^6)^2 +(\xi^7)^2 +(\xi^8 )^2,
\quad \text{
and }
\\
\hat{\phi} =\  & \xi^{1234}- \xi^{1256}+ \xi^{1278}- \xi^{1357}-\xi^{1368}- \xi^{1458} 
+ \xi^{1467}
\\\quad
&+ \xi^{2358}- \xi^{2367} -\xi^{2457}-\xi^{2468}+ \xi^{3456}- \xi^{3478}+ \xi^{5678}.
\end{align*}
Thus $\hat\phi$ is precisely the 4-form whose stabilizer defines the representation $\mathfrak{spin}(4,3) \subset \soa(4,4)$, where we use the conventions in \cite{baumkath99}. Along $\varrho=0$ the parallel form $\phi$ is equal to  $\hat\phi$.  This is sufficient to
 conclude that the stabilizer of  $\phi$ is  $\mathfrak{spin}(4,3)$, which yields
$
\mathfrak{hol}(\tilde g_{f}) \subset \mathfrak{spin}(4,3).
$
\end{proof}

The next example shows that in general the holonomy does not reduce further than $\mathfrak{po}(3,2)$, proving the inclusion 
 $\mathfrak{po}(3,2)\subset \mathfrak{spin}(4,3)$.
\begin{example} We set $f = x^1(x^3)^2$  and $Q = 0$. In this case we compute the holonomy of $\tilde{g}_{f}$ 
to be a $15$-dimensional Lie algebra $\mathfrak{hol}_p(\tilde g_{f})$. Hence it is equal to $\mathfrak{po}(3,2)$. Since we also have $\mathfrak{hol}_p(\tilde g_{f})\subset \mathfrak{spin}(4,3)$, this  shows the inclusion $\mathfrak{po}(3,2)\subset \mathfrak{spin}(4,3)$.

Note that the null $4$-plane   $\widetilde{\mathcal V}$ of  vectors $X$ such that $X \hook \alpha=0$   is the only  subspace that is invariant under the holonomy.  In particular, it does not admit any invariant lines, reflecting the fact that the conformal classes defined by $f=f(x^1,x^3)$ generically are  not conformally Einstein. 
\end{example}

Having  the inclusion $\mathfrak{po}(3,2)\subset \mathfrak{spin}(4,3)$, we can  summarize:
\begin{theorem}
Let $f=f(x^1,x^3)$ be a smooth function of two variables,   $[g_{D_f}]$ be   Bryant's conformal structure represented by a metric $g_{D_f}$ as in \eqref{bm} and with Schouten tensor $P$, and let
\[\tilde{g}_{f}=2\der(\varrho t)\der t+t^2(g_{{\mathcal D}_f}+2\,\varrho\,P),\]
be the ambient metric with $Q=0 $.
Then the holonomy of $\tilde{g}_{f}$ is contained in $\mathfrak{po}(3,2)$, and hence contained in $\spin(4,3)$.
\end{theorem}

Our last example is an ambient metric with the $7$-dimensional Heisenberg algebra as holonomy.
\begin{example} Now we set $f=x^1x^3$ and $Q= 0$. In this case it turns out that the holonomy algebra of the corresponding ambient metric is  a $7$-dimensional, $2$-step nilpotent Lie algebra
 with a $1$-dimensional centre. Hence,  as a Lie algebra, $\mathfrak{hol}_p(\tilde{g}_{f})$ is isomorphic to the $7$-dimensional Heisenberg algebra $\mathfrak{he}_3(\mathbb{R})$  defined as 
\[ \mathfrak{he}_3(\mathbb{R})\ =\ 
\left\{\left. {\small \begin{pmatrix} 0 & u & r \\
0&0&v
\\
0&0&0\end{pmatrix}}\right| u,v\in \R^3, r\in\R\right\}
.\]
In addition to the parallel forms described in Proposition \ref{propnew3}, this ambient metric admits two parallel $1$-forms, namely
\begin{eqnarray*}
\varphi^1=
(x^1)^{c_{+}}\,dt +\frac{c_+\, t\,}{9\,  (x^1)^{c_{-}}}\, \theta^4
&\text{ and }&
  \varphi^2=
(x^1)^{c_{-}}\,dt +
    \dfrac{c_-\, t}{9\, (x^1)^{c_+}}\theta^4,\quad\text{with         $c_{\pm} =  \frac{1}{6}(9 \pm \sqrt{21})$.}
        \end{eqnarray*}
The parallel $1$-forms $\varphi^1$ and $\varphi^2$ are  linearly independent on a dense open set,   null and orthogonal to each other.
Since the conformal holonomy is contained in the ambient holonomy \cite{armstrong-leistner07}, $\varphi^1$ and $\varphi^2$ define two conformal standard tractors  that are null and parallel for the normal conformal tractor connection. It is well-known (see for example \cite{leistner05a,Leitner05}) that parallel tractors, on a dense open set,  correspond to  local Einstein scales (Ricci-flat scales in case of null tractors). We conclude that the  conformal class $[g_{{\mathcal D}_f}] $ defined by  $f=x^1x^3$ contains two local Ricci-flat  metrics. Indeed,  one verifies that the metrics
\[
\left(c_1(x^1)^{c_+}-c_2(x^1)^{c_-} \right)^{-2}\ g_{{\mathcal D}_f},
\]
in which $c_1$ and $c_2$ are constants, 
are Ricci-flat.

\end{example}

\begin{remark}[Non-vanishing obstruction]
We close the paper with the observation that \emph{not all} of Bryant's conformal structures have vanishing  obstruction tensor. For  functions $f$ that are more general than $f=f(x^1,x^3)$ or $f=f(x^2,x^3)$, the Fefferman-Graham obstruction does \emph{not} vanish. This happens for example for the generic rank 3 distribution associated with the function \[f=x^3+x^1 x^2+(x^2)^2+(x^3)^2.\]
For this $f$ the obstruction tensor ${\mathcal O}_{ij}$ in \eqref{ot} is \emph{not} zero --- it has 13 non-vanishing components.
Hence, the conformal class associated to this $f$ does not admit an analytic ambient metric.
\end{remark}


\providecommand{\MR}[1]{}\def\cprime{$'$} \def\cprime{$'$} \def\cprime{$'$}

\end{document}